\documentclass{aims}
\usepackage{amsmath}
  \usepackage{paralist}
  \usepackage{graphics} %% add this and next lines if pictures should be in esp format
  \usepackage{epsfig} %For pictures: screened artwork should be set up with an 85 or 100 line screen
 \usepackage[colorlinks=true]{hyperref}
 \usepackage[usenames]{color}	
 \usepackage{import} % Pau: this is used to locate figures in a subdir
\usepackage{srcltx}	 % Pau: forward/reverse search latex<->xdvi
\usepackage{ulem}	 % Pau: for strike-out command sout

\hypersetup{urlcolor=blue, citecolor=red}

\def\currentvolume{X}
 \def\currentissue{X}
  \def\currentyear{200X}
   \def\currentmonth{XX}
    \def\ppages{X--XX}

\newtheorem{theorem}{Theorem}[section]

\newtheorem{proposition}{Proposition}

\theoremstyle{definition}
\newtheorem{definition}[theorem]{Definition}

\newcommand{\eps}{\varepsilon}

\definecolor{Red}{rgb}{1,0,0}	% Pau: this is used to color my comments

\title[Transition map and shadowing lemma] {Transition map and shadowing lemma for normally hyperbolic invariant manifolds
}%

\author[Amadeu Delshams, Marian Gidea, and Pablo Rold\'{a}n ]{Amadeu
Delshams$^\dag$, Marian Gidea$^\ddag$ and Pablo Rold\'{a}n $^\dag$  }
\subjclass{Primary,
37J40;  % Perturbations, normal forms, small divisors,
37C50; % Pseudorbits and  shadowing
37C29; %homoclinic and heteroclinic orbits
Secondary,
37B30}%Index theory, Morse-Conley indices

\keywords{Hamiltonian instability, Arnold diffusion, the three-body problem. }%

\email{Amadeu.Delshams@upc.edu }
\email{M-Gidea@neiu.edu}%
\email{Pablo.Roldan@upc.edu }

\thanks{
$^\dag$ Research of A.D. and P.R. was partially supported by MICINN-FEDER grant
MTM2009-06973 and CUR-DIUE grant 2009SGR859.}
\thanks{$^\ddag$ Research of M.G. was partially supported by NSF grants: DMS-0601016 and DMS-0635607.}

\begin{document}
\maketitle

\centerline{\scshape Amadeu Delshams}
\medskip
{\footnotesize
% please put the address of the first author
 \centerline{Departament de Matem\'atica Aplicada I, ETSEIB-UPC}
   \centerline{08028 Barcelona, Spain}
} % Do not forget to end the {\footnotesize by the sign }

\medskip

\centerline{\scshape Marian Gidea}
\medskip
{\footnotesize
 % please put the address of the second  and third author
\centerline{School of Mathematics, Institute for Advanced Study}
   \centerline{Princeton, NJ 08540, USA, and}
 \centerline{ Department of Mathematics, Northeastern Illinois University}
   \centerline{Chicago, IL 60625, USA}}

\medskip
\centerline{\scshape Pablo Rold\'{a}n}
\medskip
{\footnotesize
% please put the address of the first author
 \centerline{Departament de Matem\'atica Aplicada I, ETSEIB-UPC}

   \centerline{08028 Barcelona, Spain}
} % Do not forget to end the {\footnotesize by the sign }

\bigskip

% The name of the associate editor will be entered by an editorial staff
% "Communicated by the associate editor name" is not needed for special issue.
 \centerline{(Communicated by the associate editor name)}

%\address{Departament de Matem\'atica Aplicada I, ETSEIB-UPC, 08028 Barcelona, Spain}

%\address{Department of Mathematics, Northeastern Illinois University, Chicago, IL 60625, USA}%

%\address{Departament de Matem\'atica Aplicada I, ETSEIB-UPC, 08028 Barcelona, Spain}

%\thanks{}
%\date{}%
%\dedicatory{}%
%\commby{}%
% ----------------------------------------------------------------
\begin{abstract}
For a given a normally hyperbolic invariant manifold, whose stable and unstable
manifolds intersect transversally, we consider several tools and techniques to detect trajectories with prescribed itineraries:
the scattering map, the transition map, the method of correctly aligned windows, and the shadowing lemma. We provide an user's guide on how to apply these tools and techniques to detect unstable orbits in Hamiltonian systems. This consists in the following steps: (i) computation of the scattering map and of the transition map for a flow, (ii) reduction  to the scattering map and to the transition map, respectively, for the return map to some surface of section, (iii) construction of sequences of windows within the surface of section, with the successive pairs of windows correctly aligned, alternately, under the transition map,  and under some power of the inner map, (iv) detection of trajectories which follow closely those windows. We illustrate this strategy with two models: the large gap problem for nearly integrable Hamiltonian systems, and the the spatial circular restricted three-body problem.
\end{abstract}

\section{Introduction}

Consider a normally hyperbolic invariant manifold  for a flow or a map, and assume that the stable and unstable manifolds of the normally hyperbolic invariant manifold have a transverse intersection along a homoclinic manifold.  One can distinguish   an inner dynamics, associated to the restriction of the flow or of the map to the normally hyperbolic invariant manifold, and an outer dynamics, associated to the homoclinic orbits. There exist pseudo-orbits obtained by alternately following the inner dynamics and the outer dynamics for some finite periods of time. An important question on the  dynamics is whether there exist true orbits with similar behavior.
In this paper, we develop a toolkit of instruments and techniques to detect true orbits near a normally hyperbolic invariant manifold, that alternatively follow the inner dynamics and the outer dynamics,  for all time. Some of the tools discussed below have already been used in other works. The aim of this paper is to provide a general recipe on how to make a systematic use of these tools in general situations.

The first tool is the scattering map, which is defined on the normally hyperbolic invariant manifold, and assigns to the foot of an unstable fiber passing through a point in the homoclinic manifold, the foot of the corresponding stable fiber that passes through the same point in the homoclinic manifold. This tool, sometimes referred as the homoclinic map, has been used in \cite{Easton78,Garcia00,DelshamsLS08a}, and  subsequently refined and analyzed  in
\cite{DelshamsLS2000}. The scattering map can be defined  both in the flow case and in the map
case. In Section \ref{sec:nhim} we recall some background on normally hyperbolic invariant manifolds and Lambda Lemma. In Section \ref{sec:scatteringmap}  we describe the relationship between the scattering map for a flow and the scattering map  for  the return map to a surface of section. We note that the scattering map is defined in terms of the geometric structure, however it is not dynamically defined -- there is no actual orbit that is given by the scattering map.

The second tool that we discuss is the transition map, that actually follows the homoclinic orbits for a prescribed time. The transition map  can be computed in terms of the scattering map. Again, we will have a transition map for the flow and one for  the return map, and we will describe the relationships between them. The transition map is presented in Section \ref{sec:transition}.

The third tool is the topological method of correctly aligned windows (see \cite{GideaZ04a}), which is used to detect orbits with prescribed itineraries in a dynamical system. A window is a homeomorphic copy of a multi-dimensional rectangle, with a choice of an exit direction and of an entry direction.  A window is correctly aligned with another window if  the
image of the first window  crosses the second window all the way
through and across its exit set. This method is reviewed briefly in Section \ref{section:topological}.

The fourth tool is a shadowing lemma type of result for a normally hyperbolic invariant manifold, presented in Section \ref{sec:shadowing}.
The assumption is that a bi-infinite sequence of windows lying in the normally hyperbolic invariant manifold is given, with the consecutive pairs of windows being correctly aligned, alternately,  under the transition map (outer map), and under some power of the inner map.  The role of the windows is to approximate the location of orbits.
Then there exists a true orbit that follows closely these windows, in the prescribed order. To apply this lemma for a normally hyperbolic invariant manifold  for a map, one needs to reduce the dynamics from the continuous case to the discrete case by considering the return map to a surface of section, and construct the sequence of correctly aligned windows  for the return map. For this situation, the relationships between the scattering map for the flow and the scattering map for the return map,  and between  the transition map for the flow and the transition map for the  return map, explored in Section \ref{sec:scatteringmap} and Section \ref{sec:transition}, are useful.

A remarkable feature of these tools is that they can be used for both analytic arguments and rigorous numerical verifications. The scattering map and the transition map can be computed explicitly in concrete systems. The main advantage is that they can be used to reduce the dimensionality of the problem: from the phase space of a flow to a normally hyperbolic invariant manifold for the flow, and further to the normally hyperbolic invariant manifold for the return map to a surface of section. The shadowing lemma  also plays a key role in reducing the dimensionality of the problem: it requires the verification of topological conditions in the normally hyperbolic invariant manifold for the return map to conclude the existence of trajectories in the phase space of the flow. In numerical applications, reducing the number of dimensions of the objects computed is very crucial. The potency of these tools in numerical application is illustrated in \cite{DelshamsGR10}.

The main motivation for developing these tools resides with the instability problem for Hamiltonian systems. In the Appendix we describe  two models where the above  techniques  can be applied to show the existence of unstable orbits. The first model is the large gap problem for nearly integrable Hamiltonian systems. The second model is the  spatial circular restricted three-body problem.

In conclusion, we provide a practical recipe for finding trajectories with prescribed itineraries for a normally hyperbolic invariant manifold with the property that its stable and unstable manifolds have a transverse intersection along a homoclinic manifold:
\begin{itemize}
\item Compute the scattering map associated to the homoclinic manifold.
\item For some prescribed forward and backwards integration times, compute the corresponding transition  map.
\item If necessary,  reduce the dynamics from a flow to the return map via some surface of section. Determine the normally hyperbolic invariant manifold relative to the surface of section, and compute  the inner map -- the restriction of the return map relative to the normally hyperbolic invariant manifold.
\item Compute the scattering map and the transition map for the return map.
\item Construct windows within the normally hyperbolic invariant manifold relative to the surface of section, with the the property that the consecutive pairs of windows are correctly aligned, alternately,  under the transition map  and under some power of the inner map.
\item Apply the  shadowing lemma stated in  Theorem \ref{lem:shadowing1} to conclude  that there exist orbits that follow closely these windows.
\end{itemize}

\section{Preliminaries} \label{sec:nhim}

In this section we review the concepts of normal hyperbolicity for flows and maps,  normally hyperbolic invariant manifold for the return map to a surface of section, and we state a version of the Lambda Lemma that will be used in the subsequent sections.

\subsection{Normally hyperbolic invariant manifolds}
In this section we recall  the concept of a normally hyperbolic
invariant manifolds for a map and for a flow, following
\cite{Fenichel74,HirschPS77}.

Let $M$ be a $C^r$-smooth, $m$-dimensional manifold (without boundary), with $r\geq 1$,
and  $\Phi:M\times \mathbb{R}\to M$ a $C^r$-smooth flow
on $M$.

\begin{definition} A submanifold  (possibly with boundary) $\Lambda$ of $M$ is said to
be a normally hyperbolic invariant manifold for $\Phi$ if
$\Lambda$ is invariant under $\Phi$, there exists a splitting of
the tangent bundle of $TM$ into sub-bundles
\[TM=E^u\oplus E^s\oplus T\Lambda,\]
that are invariant under $d\Phi^t$ for all $t\in\mathbb{R}$, and there exist a constant $C>0$
and rates $0<\beta<\alpha$, such that for all $x\in\Lambda$ we have
\[\begin{split}
v\in E^s_x  \Leftrightarrow \|D\Phi^t(x)(v)\|\leq Ce^{-\alpha t}\|v\|  \textrm{ for all } t\geq 0,\\
v\in E^u_x  \Leftrightarrow \|D\Phi^t(x)(v)\|\leq Ce^{\alpha t}\|v\|  \textrm{ for all } t\leq 0,\\
v\in T_x\Lambda \Leftrightarrow \|D\Phi^t(x)(v)\|\leq Ce^{\beta|t|}
\|v\| \textrm{ for all } t\in\mathbb{R}.
\end{split}\]
\end{definition}

It follows that there exist stable and unstable manifolds
of $\Lambda$, as well as stable and unstable manifolds of each
point $x\in\Lambda$, which are defined by
\[\begin{split}
W^s(\Lambda)=&\{y\in M\,|\, d(\Phi^t(y),\Lambda)\leq C_ye^{-\alpha t} \textrm{ for all } t\geq 0\},\\
W^u(\Lambda)=&\{y\in M\,|\, d(\Phi^t(y),\Lambda)\leq C_ye^{\alpha t} \textrm{ for all } t\leq 0\},\\
W^s(x)=&\{y\in M\,|\, d(\Phi^t(y),\Phi^t(y)(x))\leq C_{x,y}e^{-\alpha t}\textrm{ for all } t\geq 0\},\\
W^u(x)=&\{y\in M\,|\, d(\Phi^t(y),\Phi^t(y)(x))\leq C_{x,y}e^{\alpha
t} \textrm{ for all } t\leq 0\},
\end{split}\]
for some constants $C_y,C_{x,y}>0$.

The stable and unstable manifolds of $\Lambda$ are foliated by
stable and unstable manifolds of points, respectively, i.e.,
$W^s(\Lambda)=\bigcup_{x\in\Lambda}W^s(x)$ and
$W^u(\Lambda)=\bigcup_{x\in\Lambda}W^u(x)$.

In the sequel we will assume that $\Lambda$ is a compact and
connected manifold. With no other assumptions,  $E^s_x$ and $E^u_x$ depend
continuously (but non-smoothly) on $x\in M$; thus the dimensions of $E^s_x$ and $E^u_x$
are independent of $x$. Below we only consider the case when
the dimensions of the stable and unstable bundles are equal. We
denote $n=\dim(E^s_x)=\dim(E^u_x)$, $l=\dim(T_x\Lambda)$, where
$2n+l=m$.

The smoothness of the invariant objects defined by the normally
hyperbolic structure depends on the rates $\alpha$ and $\beta$. Let
 $\ell$ be a positive integer satisfying $1\leq \ell<\min\{r,
\alpha/\beta\}$. The manifold $\Lambda$ is
$C^\ell$-smooth. The stable and unstable manifolds $W^s(\Lambda)$
and $W^u(\Lambda)$ are $C^{\ell-1}$-smooth. The splittings $E^s_x$
and $E^u_x$ depend $C^{\ell-1}$-smoothly on $x$. The stable and
unstable fibers $W^s(x)$ and $W^u(x)$ are $C^r$-smooth. The stable and
unstable fibers $W^s(x)$ and $W^u(x)$ depend $C^{\ell-1-j}$-smoothly
on $x$ when $W^s(x),W^u(x)$ are endowed with the $C^j$-topology. In
the sequel we will assume that the rates are such that there exists
an integer $\ell\geq 2$ as above, and that  all the manifolds and maps considered below are at least
$C^k$-smooth, with $2\leq k\leq \ell$.

The notion of normal hyperbolicity for maps is very similar. Let $F:M\to M$ be a $C^r$-smooth map on $M$.

\begin{definition} A submanifold $\Lambda$ of $M$ is said to
be a normally hyperbolic invariant manifold for $F$ if $\Lambda$ is
invariant under $F$, there exists a splitting of the tangent bundle
of $TM$ into sub-bundles
\[TM=E^u\oplus E^s\oplus T\Lambda,\]
that are invariant under $dF$, and there exist a constant $C>0$ and
rates $0<\lambda<\mu^{-1}<1$, such that for all $x\in\Lambda$ we
have
\[\begin{split}
v\in E^s_x  \Leftrightarrow \|DF^k_x(v)\|\leq C\lambda^k\|v\|  \textrm{ for all } k\geq 0,\\
v\in E^u_x  \Leftrightarrow \|DF^k_x(v)\|\leq C\lambda^{-k}\|v\|  \textrm{ for all } k\leq 0,\\
v\in T_x\Lambda \Leftrightarrow \|DF^k_x(v)\|\leq C\mu^{|k|}\|v\|
\textrm{ for all } k\in\mathbb{Z}.
\end{split}\]
\end{definition}

There exist stable and unstable manifolds of $\Lambda$, as well as
the stable and unstable manifolds of each point $x\in\Lambda$, that are defined similarly as in the flow case, and they carry analogous properties. The smoothness properties of these invariant objects are analogous of those for a flow, if we set $1\leq \ell<\min\{r, (\log\lambda^{-1})(\log\mu)^{-1}\}$.

\subsection{Normal hyperbolicity relative to the return map}\label{subsection:normalpoincare}

Let  $\Phi:M\times \mathbb{R}\to M$ be a $C^r$-smooth flow defined on
an $m$-dimensional manifold $M$. Denote by  $X$  the vector field
associated to $\Phi$, where $X(x)=\frac{\partial}{\partial t}
\Phi(x,t)_{\mid t=0}$.  As before, assume that $\Lambda\subseteq M$ is
an $l$-dimensional normally hyperbolic invariant manifold for $\Phi$.
The dimensions of $T\Lambda$, $E^u$ and $E^s$ are $l,n,n$,
respectively, with $l+2n=m$.

Let $\Sigma$ be an $(m-1)$-dimensional local surface of section, i.e., $\Sigma$ is a $C^1$ submanifold of $M$ such that  $X(x)\not\in T_x\Sigma$ for all $x\in\Sigma$. Let $\Lambda_\Sigma=\Lambda\cap\Sigma$. Then $\Lambda_\Sigma$ is a $(l-1)$-dimensional submanifold in $\Sigma$, assuming that the intersection is non-empty.

Assume that each forward and backward orbit through a point in $\Lambda_\Sigma$ intersects again $\Lambda_\Sigma$. Since $X(x)\not\in T_x\Sigma$ for all $x\in\Sigma$, then  the intersection of the forward and backward orbits with $\Sigma$ are transverse. Also, $X(x)\not\in T_x\Lambda_\Sigma$ for all $x\in\Lambda_\Sigma$. Additionally,  assume that the function \[\tau:\Lambda_\Sigma\to (0,\infty), \textrm { given by } \tau(x)=\inf\{ t>0\,|\, \Phi(x,\tau(x))\in \Lambda_\Sigma\},\] is a continuous function. Following \cite{Fenichel74}, we will refer to $\Lambda_\Sigma$ with these properties as a thin surface of section.

By the Implicit Function Theorem, $\tau$ can be extended to a $C^1$-smooth function in a neighborhood $U_\Sigma$ of $\Lambda_\Sigma$ in $\Sigma$ such that $\Phi(x,\tau(x))\in\Sigma$ for all $x\in U_\Sigma$. The Poincar\'e first return map to $\Sigma$ is the map $F:U_\Sigma \to \Sigma$ given by $F(x)=\Phi^{\tau(x)}(x)$.

Let $\Lambda_\Sigma ^X\subseteq \Lambda$ be the union of the orbits of the flow through points in $\Lambda_\Sigma$. Since $\Lambda_\Sigma^X$ is a $C^1$-submanifold of $\Lambda$, and is invariant under $\Phi$, then is a normally hyperbolic invariant manifold for the flow $\Phi$.  The theorem below  implies  that the manifold $\Lambda_\Sigma$ is normally hyperbolic for the  return map $F$.
\begin{theorem}{(Fenichel, \cite{Fenichel74})}\label{thm:normalpoincare}
Let $\Lambda_\Sigma $ be a thin surface of section for the vector field
$X$ on $M$. Then $\Lambda_\Sigma$ is normally hyperbolic with respect to
$F$ if and only if $\Lambda_\Sigma^X$ is a normally
hyperbolic invariant manifold with respect to $\Phi$.
\end{theorem}
The invariant sub-bundles $T\Lambda$, $E^u$, $E^s$ associated to the normal hyperbolic structure on $\Lambda$ correspond to sub-bundles $T\Lambda_\Sigma $, $E^u_\Sigma $, $E^s_\Sigma $ in  the following way. Let $\pi:TM=\textrm{span}(X)\oplus T\Sigma\to T\Sigma$ be the projection onto $T\Sigma$. Then
$T\Lambda_\Sigma  =\pi(T_\Lambda)$, $E^u_\Sigma=\pi(E^u)$, and $E^s_\Sigma=\pi(E^s)$.

Note that the surface of section described above is only a local surface of section. It is in general very difficult, or even impossible, to obtain a global surface of section for a flow. However, one can obtain global surfaces of section for Hamiltonian flows on $3$-dimensional strictly convex energy surfaces \cite{HWZ98}.

\subsection{Lambda Lemma}\label{sec:lambda}

We describe a Lambda Lemma type of result for normally hyperbolic
invariant manifolds that appears in J.-P. Marco \cite{Marco08}.

By a normal form in a neighborhood $V$ of $\Lambda$ in
$M$ we mean a $C^k$-smooth coordinate system $(c,s,u)$ on $V$ such
that $V$ is diffeomorphic  through $(c,s,u)$ with a product $\Lambda\times \mathbb{R}^n\times \mathbb{R}^n$, where $\Lambda=\{(c,s,u)\,|\, c\in \Lambda,\, u=s=0\}$, and $W^u(x)=\{(c,s,u)\,|\,c=c(x), s=0\}$,
$W^s(x)=\{(c,s,u)\,|\, c=c(x), u=0\}$ for each $x\in\Lambda$ of coordinates $(c(x),0,0)$.

\begin{theorem}[Lambda Lemma]\label{thm:lambdalemma}
Suppose that $\Lambda$ is a normally hyperbolic invariant manifold for $F$ and $(c,s,u)$
is a normal form in a neighborhood of $\Lambda$.
Consider a submanifold $\Delta$ of $M$ of dimension $n$ which intersects the stable manifold
$W^s(\Lambda)$ transversely at some point $z= (c, s, 0)$. Set $F^N(z)=z_N=(c_N, s_N, 0)$ for
$N\in\mathbb{N}$.
Then there exists $\delta>0$ and $N_0>0$ such that for each $N\geq N_0$ the connected component $\Delta_N$ of $F^N(\Delta)$ in the $\delta$-neighborhood $V(\delta)=\Lambda\times B^s_{\delta}(0)\times B^u_{\delta}(0)$ of $\Lambda$ in $M$ admits a graph parametrization  of the form
\[\Delta_N:=\{(C_N(u), S_N(u),u) \,|\, u\in B^u_{\delta}(0)\}\]
such that
\[\|C_N-c_N\|_{C^1(B^u_{\delta}(0))}\to 0, \, \textrm { and } \|S_N\|_{C^1(B^u_{\delta}(0))}\to 0 \textrm { as } N\to\infty.  \]
\end{theorem}

\section{Scattering map} \label{sec:scatteringmap}

In this section we review the scattering map associated to a normally hyperbolic invariant manifold for a flow or for a map, and discuss the relationship between the scattering map for a flow and the scattering map for the corresponding return map to some surface of section.

\subsection{Scattering map for continuous and discrete dynamical systems} Consider a flow $\Phi: M\times \mathbb{R}\to M$ defined on a manifold $M$ that possesses a normally hyperbolic invariant manifold $\Lambda\subseteq M$.

As the stable and unstable manifolds of $\Lambda$ are foliated by stable and unstable manifolds of points, respectively, for each $x\in W^u(\Lambda)$ there exists a unique $x_-\in\Lambda$ such that
$x\in W^u(x_-)$, and for each   $x\in W^s(\Lambda)$ there exists
a unique $x_+\in\Lambda$ such that $x\in W^s(x_+)$.
We define the  wave maps  $\Omega_{+}:W^s(\Lambda)\to \Lambda$ by
$\Omega_{+}(x)=x_{+}$, and $\Omega_{-}:W^u(\Lambda)\to \Lambda$  by
$\Omega_{-}(x)=x_{-}$. The maps $\Omega _+$ and $\Omega _-$ are
$C^{\ell}$-smooth.

We now describe the scattering  map, following \cite{DelshamsLS08a}.  Assume that $W^u(\Lambda)$ has a transverse intersection with $W^s(\Lambda)$ along a $l$-dimensional homoclinic manifold $\Gamma$. The manifold $\Gamma$ consists of a $(l-1)$-dimensional family of trajectories asymptotic to $\Lambda$ in both forward and backwards time.
The transverse intersection of the hyperbolic invariant manifolds along $\Gamma$ means that  $\Gamma\subseteq W^u(\Lambda) \cap W^s(\Lambda)$
and, for each $x\in\Gamma$, we have
\begin{equation} \begin{split}\label{eq:dynamical channel}
T_xM=T_xW^u(\Lambda)+T_xW^s(\Lambda),\\
T_x\Gamma=T_xW^u(\Lambda)\cap T_xW^s(\Lambda).
\end{split} \end{equation}
Let us assume the additional condition that for each $x\in\Gamma$ we
have
\begin{equation} \label{eq:transverse foliation}
\begin{split}
T_xW^s(\Lambda)=T_xW^s(x_+)\oplus T_x(\Gamma),\\
T_xW^u(\Lambda)=T_xW^u(x_-)\oplus T_x(\Gamma),
\end{split}
\end{equation}
where $x_-,x_+$ are the uniquely defined points in $\Lambda$
corresponding to $x$.

The restrictions $\Omega_+^\Gamma,\Omega_-^\Gamma$ of 
$\Omega_+,\Omega_-$, respectively, to $\Gamma$ are  \emph{local}
$C^{\ell-1}$ - diffeomorphisms.  By restricting $\Gamma$ even further,  if necessary,  we can ensure that
$\Omega_+^\Gamma,\Omega_-^\Gamma$ are $C^{\ell-1}$-diffeomorphisms.
A homoclinic manifold $\Gamma$ for which the corresponding
restrictions of the wave maps are $C^{\ell-1}$-diffeomorphisms will
be referred as a homoclinic channel.
\begin{definition} \label{defn:scattering_map_flow}
Given a homoclinic channel $\Gamma$, the scattering map associated to
$\Gamma$ is the $C^{\ell-1}$-diffeomorphism
$S^\Gamma=\Omega^\Gamma_+\circ (\Omega^\Gamma_-)^{-1}$ defined on
the open subset $U_-:=\Omega^\Gamma_-(\Gamma)$ in $\Lambda$ to the
open subset $U_+:=\Omega^\Gamma_+(\Gamma)$ in $\Lambda$.
\end{definition}

See Figure \ref{fig:scattering2}. In the sequel we will regard $S^\Gamma$ as a partially defined map, so the
image of a set $A$ by $S^\Gamma$ is $S^\Gamma(A\cap U_-)$.
\begin{figure}
\centering
\includegraphics[width=0.5\textwidth]{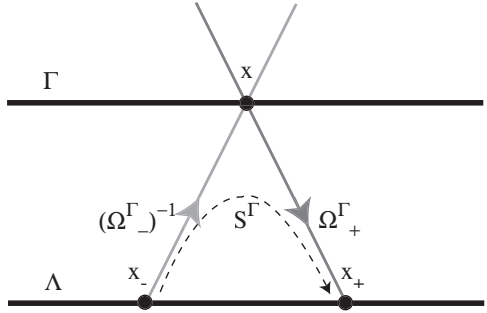}
\caption{Scattering map.}
\label{fig:scattering2}
\end{figure}

If we flow $\Gamma$ backwards and forward in time we obtain
the manifolds $\Phi^{-t_u}(\Gamma)$ and
$\Phi^{t_s}(\Gamma)$ that  are also homoclinic channels, where
 $t_u,t_s>0$. The associated wave maps
are $\Omega_+^{\Phi^{-t_u}(\Gamma)},\Omega_-^{\Phi^{-t_u}(\Gamma)}$, and $\Omega_+^{\Phi^{t_s}(\Gamma)},\Omega_-^{\Phi^{t_s}(\Gamma)}$,
respectively. The
scattering map can be expressed with respect to these wave maps as
\begin{equation}\label{eqn:scatteringexplicit flow}
S^\Gamma=\Phi^{-t_s}\circ (\Omega^{\Phi^{t_s}(\Gamma)}_+)\circ \Phi^{t_s+t_u}\circ
(\Omega ^{\Phi^{-t_u}(\Gamma)}_-)^{-1}\circ \Phi^{-t_u}.
\end{equation}

We recall below some important properties of the scattering map.

\begin{proposition}\label{prop:symplectic}Assume that $\dim M=2n+l$ is even (i.e., $l$ is even) and $M$ is endowed with a symplectic (respectively exact symplectic) form $\omega$ and that $\omega_{\mid \Lambda}$  is also symplectic.
Assume that $\Phi^t$ is symplectic (respectively exact symplectic).
Then, the scattering map $S^\Gamma$ is symplectic (respectively exact symplectic).\end{proposition}

\begin{proposition}\label{prop:transversal}Assume that $T_1$ and $T_2$ are two invariant  submanifolds of complementary dimensions in $\Lambda$. Then  $W^u(T_1)$ has a transverse intersection with $W^s(T_2)$ in $M$ if and only if $S(T_1)$ has a transverse intersection with $T_2$ in $\Lambda$.\end{proposition}

In the case of  a discrete dynamical system consisting of a diffeomorphism $F:M\to M$ defined on a manifold $M$, the scattering map is defined in a similar way. We assume that $F$ has a normally hyperbolic invariant manifold $\Lambda\subseteq M$.
The wave maps are defined by $\Omega_{+}:W^s(\Lambda)\to \Lambda$ with
$\Omega_{+}(x)=x_{+}$, and $\Omega_{-}:W^u(\Lambda)\to \Lambda$ with
$\Omega_{-}(x)=x_{-}$.

Assume that $W^u(\Lambda)$ and
$W^s(\Lambda)$ have a differentiably transverse intersection along a
homoclinic $l$-dimensional $C^{\ell-1}$-smooth manifold $\Gamma$.
We also assume the transverse foliation condition
\eqref{eq:transverse foliation}.

A homoclinic manifold $\Gamma$ for which the corresponding
restrictions of the wave maps are $C^{\ell-1}$-diffeomorphisms is referred as a homoclinic channel.
\begin{definition} \label{defn:scattering_map_map}
Given a homoclinic channel $\Gamma$, the scattering map associated to
$\Gamma$ is the $C^{\ell-1}$-diffeomorphism
$S^\Gamma=\Omega^\Gamma_+\circ (\Omega^\Gamma_-)^{-1}$ defined on
the open subset $U_-:=\Omega^\Gamma_-(\Gamma)$ in $\Lambda$ to the
open subset $U_+:=\Omega^\Gamma_+(\Gamma)$ in $\Lambda$.
\end{definition}

Note that for $M,N>0$, the manifolds $F^{-M}(\Gamma)$ and
$F^N(\Gamma)$ are also homoclinic channels. The associated wave maps
are $\Omega_-^{F^{-M}(\Gamma)},\Omega_+^{F^{-M}(\Gamma)}$, and $\Omega_-^{F^{N}(\Gamma)},\Omega_+^{F^{N}(\Gamma)}$. The
scattering map can be expressed with respect to these wave map as
\begin{equation}\label{eqn:scatteringexplicit}
S^\Gamma=F^{-N}\circ (\Omega^{F^N(\Gamma)}_+)\circ F^{M+N}\circ
(\Omega ^{F^{-M}(\Gamma)}_-)^{-1}\circ F^{-M}.
\end{equation}

The scattering map for the discrete case satisfies symplectic and transversality properties similar to those in Proposition \ref{prop:symplectic} and Proposition \ref{prop:transversal} for the continuous case.

\subsection{Scattering map for the return map}
\label{subsection:scatteringreturn}

Let  $\Phi:M\times \mathbb{R}\to M$ be a $C^r$-smooth flow defined on
an $m$-dimensional manifold $M$, and   $X$   be the vector field
associated to $\Phi$. Let $\Lambda\subseteq M$ be an $l$-dimensional
normally hyperbolic invariant manifold for $\Phi$. Assume that
$\Sigma$ is a local surface of section and
$\Lambda_\Sigma=\Lambda\cap\Sigma$ satisfies the conditions in
Subsection \ref{subsection:normalpoincare}.

Consider $\Gamma$ a homoclinic channel for $\Phi$. First, we assume that $\Gamma$ has a non-empty intersection with $\Sigma$. Note that $\Gamma$ is a $(l-1)$-parameter family of orbits; we further assume that each trajectory intersects $\Sigma$  transversally.  Since $\Gamma$ is a homoclinic channel, each orbit intersects $\Sigma$ exactly once. Let $\Gamma_\Sigma=\Gamma\cap \Sigma$. It is easy to see that $\Gamma_\Sigma$ is a homoclinic channel for $F$. Thus, we have a scattering map $S^\Gamma$ for $\Gamma$ associated to the flow $\Phi$, and we also have a scattering map $S^{\Gamma_\Sigma}$ for $\Gamma_\Sigma$ associated to the map $F$.

We want to understand the relationship between $S^\Gamma$ and  $S^{\Gamma_\Sigma}$. Associated to the homoclinic channels $\Gamma$ and $\Gamma_\Sigma$ there exist  wave maps $\Omega_\pm^{\Gamma}: \Gamma\to\Lambda$ and $\Omega_\pm^{\Gamma_\Sigma}: \Gamma_\Sigma\to\Lambda_\Sigma$, respectively. These maps are diffeomorphisms. Let $x\in\Gamma_\Sigma$, and let $x_-=\Omega_-^{\Gamma}(x)$, $x_+=\Omega_+^{\Gamma}(x)$, and $\hat x_-=\Omega_-^{\Gamma_\Sigma}(x)$, $\hat x_+=\Omega_+^{\Gamma_\Sigma}(x)$. We have $S^\Gamma(x_-)=x_+$ and   $S^{\Gamma_\Sigma}(\hat x_-)=\hat x_+$.

We want to relate $x_-$ with $\hat x_-$, and $x_+$ with $\hat x_+$.
These points are all in $\Lambda$. It is clear that $\hat
x_-=\Omega^{\Gamma_\Sigma}_-\circ (\Omega^{\Gamma}_-)^{-1}(x_-)$, and
$\hat
x_+=\Omega^{\Gamma_\Sigma}_+\circ
(\Omega^{\Gamma}_+)^{-1}(x_+)$. Denote by
$P^{\Gamma}_-:\Omega^{\Gamma}_-(\Gamma)\to \Omega^{\Gamma_\Sigma}_-$
the map given by $P^{\Gamma}_-=\Omega^{\Gamma_\Sigma}_-\circ
(\Omega^{\Gamma}_-)^{-1}$, and denote by
$P^{\Gamma}_+:\Omega^{\Gamma}_+(\Gamma)\to \Omega^{\Gamma_\Sigma}_+$
the map given by $P^{\Gamma}_+=\Omega^{\Gamma_\Sigma}_+\circ
(\Omega^{\Gamma}_+)^{-1}$. We want to express these maps in terms of
the dynamics restricted to $\Lambda$.

Let  $V$  be a flow box at  $\hat x_-$ (for definition see \cite{Robinson1999}). This means each trajectory through a point $y\in V$ intersects $\Sigma$ exactly once. Then there exists a differentiable function $\hat \tau:V\to\mathbb{R}$ defined by $\tau(z)=0$ if $z\in\Sigma$ and $\Phi^{\hat\tau(y)} (y)\in\Sigma$  for each $y\in V$.
The function $\hat\tau$ can be extended in a unique way on each trajectory passing though $V$.
Due to the relationship between the invariant bundles for the flow and the invariant bundles for the map described in  Subsection \ref{subsection:normalpoincare}, the fiber $E^u_{\Sigma}(\hat x_-)$  is the projection onto $T\Sigma$ of the image of the fiber $E^u(x_-)$ under $D\Phi^{\hat\tau(x_-)}_{x_-}$. This means that $\Phi^{\hat\tau (x_-)}(x_-)=\hat x_-$. In other words, $\hat x_-$ is at the intersection of the trajectory through $x_-$ with $\Sigma$. Thus, the projection $P^{\Gamma}_-$ that takes $x_-$ to $\hat x_-$ is given by $P^{\Gamma}_-(x_-)=\Phi^{\hat\tau (x_-)}(x_-)$. This projection map is invertible. If $\hat y_-$ is a point in $\Omega^{\Gamma_\Sigma}_-$,  there exists a unique point $y_-\in \Omega^{\Gamma}_-(\Gamma)$ such that $\Phi^{\hat\tau (y_-)}(y_-)=\hat y_-$. If there exist two such points, $y_-$ and $y'_-$, to them they correspond two points $y$, $y'$ in $\Gamma_\Sigma$ such that $y\in W^u_{F}(y_-)$ and $y'\in W^u_{F}(y'_-)$. The points $y,y'$ should belong to the same unstable fiber $W^u_{F}(\hat y_-)$. Then it means that $y$, $y'$ are on the same trajectory. As they are also in $\Gamma$ and $\Gamma$ is a homoclinic channel, than $y=y'$ and $y_-=y'_-$. In summary, the projection map $P^{\Gamma}_-:\Omega^{\Gamma}_-(\Gamma)\to \Omega^{\Gamma_\Sigma}_-$ is given by $P^{\Gamma}_-(x_-)=\Phi^{\tau(x_-)}(x_-)$.
Similarly, the projection map $P^{\Gamma}_+:\Omega^{\Gamma}_+(\Gamma)\to \Omega^{\Gamma_\Sigma}_+$ is given by $P^{\Gamma}_+(x_+)=\Phi^{\tau(x_+)}(x_+)$. See Figure \ref{fig:scattering-returnmap}.

\begin{figure}
\centering
\includegraphics[width=0.5\textwidth]{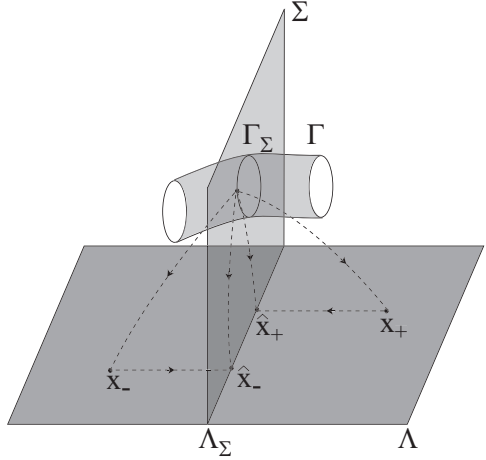}
\caption{Scattering map for the return map.}
\label{fig:scattering-returnmap}
\end{figure}

Now we can formulate the relationship between the scattering map $S^\Gamma$  associated to the flow $\Phi$, and the scattering map $S^{\Gamma_\Sigma}$ associated to the map $F$.
\begin{proposition}\label{prop:scatteringmapflow1} Assume that $\Gamma$ is a homoclinic channel for the flow $\Phi$, and $\Gamma_\Sigma=\Gamma\cap \Sigma$ is the corresponding homoclinic channel for the map $F$. Let $S^\Gamma$ be the scattering map corresponding to $\Gamma$, and let $S^{\Gamma_\Sigma}$ be the scattering map  corresponding to $\Gamma_\Sigma$. Then:
\begin{equation}
S^{\Gamma_\Sigma}=P_+^\Gamma\circ S^{\Gamma}\circ (P_-^\Gamma)^{-1}.
\end{equation}
\end{proposition}
\begin{proof} We have that $S^{\Gamma}(x_-)=x_+$, $S^{\Gamma_\Sigma}(\hat x_-)=\hat x_+$, $P^{\Gamma}_-(x_-)=\hat x_-$, and $P^{\Gamma}_-(x_+)=\hat x_+$. Thus $S^{\Gamma_\Sigma}(\hat x_-)=P^{\Gamma}_+(x_+)=P^{\Gamma}_+\circ S^{\Gamma}(x_-)=P^{\Gamma}_+\circ S^{\Gamma}\circ (P^{\Gamma}_-)^{-1}(\hat x_-)$.
\end{proof}

\section{Transition map}\label{sec:transition}

The scattering map for a flow $\Phi$ is geometrically defined:   $S^\Gamma(x_-)=x_+$ means that $W^u(x_-)$ intersects $W^s(x_+)$ at a unique point $x\in\Gamma$, with $W^u(x_-)$ and $W^s(x_+)$ being $n$-dimensional manifolds. However, there is no trajectory of the system that goes from near $x_-$ to near $x_+$. Instead, the trajectory of $x$ approaches asymptotically the backwards orbit of $x_-$ in negative time, and approaches asymptotically the forward orbit of $x_+$ in positive time. For applications we need
a dynamical version of the scattering map. That is, we need a map that takes some backwards image of $x_-$ into some forward image of $x_+$. We will call this map a transition map. The transition map depends on the amounts of times
we want to flow in the past and in the future.  The transition map carries the same geometric information as the scattering map. Since in perturbation problems the scattering map can be computed explicitly, the transition map is also computable. The notion of transition map below is similar to the transition map defined in \cite {Cresson2008}, however, their version is not related to the scattering map.

\subsection{Transition  map for continuous and discrete dynamical systems}
Consider a flow $\Phi: M\times \mathbb{R}\to M$ defined on a manifold $M$ that possesses a normally hyperbolic invariant manifold $\Lambda\subseteq M$. Assume that $W^u(\Lambda)$ and $W^s(\Lambda)$ have a transverse intersection, and that there exists a homoclinic channel $\Gamma$. Given $t_u,t_s>0$,
the time-map $\Phi^{t_s+t_u}$ is a diffeomorphism from
$\Phi^{-t_u}(\Gamma)$ to $\Phi^{t_s}(\Gamma)$. Using
\eqref{eqn:scatteringexplicit flow} we can express the restriction of
$\Phi^{t_s+t_u}$ to $\Phi^{-t_u}(\Gamma)$  in terms of the scattering
map as \[\Phi^{t_s+t_u}_{\mid
\Phi^{-t_u}(\Gamma)}:(\Omega^{\Phi^{-t_u}(\Gamma)}_-)^{-1}(\Phi^{-t_u}(U_-))
\to (\Omega^{\Phi^{t_s}(\Gamma)}_+)^{-1}(\Phi^{t_s}(U_+)), \]  given by
\begin{equation}
\Phi^{t_s+t_u}_{\mid \Phi^{-t_u}(\Gamma)}=(\Omega ^{\Phi^{t_s}(\Gamma)}_+)^{-1}\circ \Phi^{t_s}\circ
S^\Gamma\circ \Phi^{t_u}\circ (\Omega ^{\Phi^{-t_u}(\Gamma)}_-),
\end{equation}
where $S^\Gamma:U_-\to U_+$ is the scattering map associated to the
homoclinic channel $\Gamma$. We use this to define the transition map as an  an approximation of $\Phi^{t_s+t_u}_{\mid \Phi^{-t_u}(\Gamma)}$ provided that $t_u,t_s$ are sufficiently large.

\begin{definition} \label{defn:transition_map flow} Let $\Gamma$ be a homoclinic channel for $\Phi$. Let $t_u, t_s>0$ fixed.
The transition map $S_{t_u,t_s}^{\Gamma}$ is a diffeomorphism
\[S_{t_u,t_s}^{\Gamma}:\Phi^{-t_u}(U_-)
\to  \Phi^{t_s}(U_+)  \] given by
\begin{equation*}\label{eqn:transitiondefn flow}
S_{t_u,t_s}^{\Gamma}= \Phi^{t_s}\circ
S^\Gamma\circ \Phi^{t_u} ,
\end{equation*}
where $S^\Gamma:U_-\to U_+$ is the scattering map associated to the
homoclinic channel $\Gamma$.
\end{definition}

Alternatively, we can express the transition map as
\begin{equation*}
S_{t_u,t_s}^{\Gamma}=\Omega^{\Phi^{t_s(\Gamma)}}_+\circ \Phi^{t_u+t_s}\circ(\Omega^{\Phi^{-t_u(\Gamma)}}_-)^{-1}
\end{equation*}

The symplectic property and the transversality property of the scattering map lend themselves to similar properties of the
transition map.

In the case of a dynamical system given by a map $F:M\to M$, the transition map can be defined in a similar manner
to the flow case, and enjoys similar properties.
As before, we assume that $\Lambda\subseteq M$ is a normally hyperbolic invariant manifold for $F$. \begin{definition} \label{defn:transition_map map} Let $\Gamma$ be a homoclinic channel for $F$. Let $N_u, N_s>0$ fixed.
The transition map $S_{N_u,N_s}^{\Gamma}$ is a diffeomorphism
\[S_{N_u,N_s}^{\Gamma}:F^{-N_u}(U_-)
\to  F^{N_s}(U_+)  \] given by
\begin{equation*}
S_{N_u,N_s}^{\Gamma}= F^{N_s}\circ
S^\Gamma\circ F^{N_u} ,
\end{equation*}
where $S^\Gamma:U_-\to U_+$ is the scattering map associated to the
homoclinic channel $\Gamma$.
\end{definition}

\subsection{Transition map for the return map}

We will consider the reduction of the transition map to a local surface of section. Let $\Sigma$ be a local surface of section and $\Lambda_\Sigma=\Lambda\cap \Sigma$. By Theorem \ref{thm:normalpoincare}, $\Lambda_\Sigma$ is normally hyperbolic with respect to the first return map to $\Sigma$.
Assume that $\Gamma$ intersects $\Sigma$ as in Subsection \ref{subsection:normalpoincare}, and let $\Gamma_\Sigma=\Gamma\cap\Sigma$.

Let $x$ be a point in $\Gamma_\Sigma$.
Then $\Phi^{-t_u}(x)$ lies on $W^u(\Phi^{-t_u}(x_-))$, approaches asymptotically $\Lambda$ as $t_u\to \infty$,  and intersects $\Sigma$ infinitely many times. Similarly,  $\Phi^{t_s}(x)$ lies on $W^u(\Phi^{t_s}(x_+))$, approaches asymptotically $\Lambda$ as $t_s\to \infty$,  and intersects $\Sigma$ infinitely many times.

We want to choose and fix some times $t_u,t_s$, depending on $x\in\Gamma$, such that $\Phi^{-t_u}(x), \Phi^{t_s}(x)$ are both in $\Sigma$, and moreover, $\Phi^{-t_u}(x), \Phi^{t_s}(x)$ are sufficiently close to $\Phi^{-t_u}(x_-), \Phi^{t_s}(x_+)$, respectively.

Let $\upsilon>0$ be a small positive number. We define $t_u=t_u(x)$ to be the smallest time such that $\Phi^{-t_u(x)}(x) \in \Sigma$, and the distance between $\Phi^{-t_u}(x)$ and $\Phi^{-t_u}(x_-)$, measured along the unstable fiber $W^u(\Phi^{-t_u}(x_-))$, is less than $\upsilon$. Let $N_u>0$ be such that $\Phi^{-t_u}(x)=F^{-N_u}(x)$.
Similarly, we define $t_s=t_s(x)$ to be the smallest time such that $\Phi^{t_s}(x) \in\Sigma$, and the distance between $\Phi^{t_s}(x)$ and $\Phi^{t_s}(x_+)$, measured along the stable fiber $W^s(\Phi^{t_s}(x_+))$, is less than $\upsilon$. Let $N_s>0$ be such that $\Phi^{t_s}(x)=F^{N_S}(x)$.

At this point, we have a transition map $S^{\Gamma}_{t_u,t_s}$ associated to the flow $\Phi$ and to the homoclinic channel $\Gamma$ for the flow, and a transition map $S^{\Gamma_\Sigma}_{N_u,N_s}$ associated to the map $F$ and to the homoclinic channel $\Gamma_\Sigma$ for the map.

We have that $\Phi^{-t_u}(\Gamma)$ and $\Phi^{t_s}(\Gamma)$ are both homoclinic channels for the flow $\Phi$, and
$F^{-N_u}(\Gamma)$ and $F^{N_s}(\Gamma)$ are both homoclinic channels for the map $F$.  Let us consider the projection mappings $P_-^{F^{-N_u}}(\Gamma), P_+^{F^{-N_u}}(\Gamma) $ associated to the homoclinic channel $F^{-N_u}(\Gamma)$, and
the projection mappings $P_-^{F^{N_s}}(\Gamma), P_+^{F^{N_s}}(\Gamma) $ associated to the homoclinic channel $F^{N_s}(\Gamma)$. These projections mappings are defined as in Subsection \ref{subsection:normalpoincare}.

The relationship between the transition map for the flow $\Phi$ and the transition map for the return map $F$ is given by the following:

\begin{proposition} \label{prop:transitionmapflow1} Assume that $\Gamma$ is a homoclinic channel for the flow $\Phi$, and $\Gamma_\Sigma=\Gamma\cap \Sigma$ is the corresponding homoclinic channel for the map $F$. Let $t_u$, $t_s$, $N_u$, $N_s>0$ be fixed.
Let $S^\Gamma_{N_u,N_s}$ be the transition map corresponding to $\Gamma$ for the flow $\Phi$, and let $S_{N_u,N_s}^{\Gamma_\Sigma}$ be the transition map  corresponding to $\Gamma_\Sigma$ for the return map $F$.
Then
\[S^{\Gamma_\Sigma}_{N_u,N_s}=P^{F^{N_s}(\Gamma)}_+\circ S^{\Gamma}_{t_u,t_s}\circ (P^{F^{-N_u(\Gamma)}}_-)^{-1}.\]
\end{proposition}
\begin{proof}
We have that $S^{\Gamma_\Sigma}(\hat x_-)=\hat x_+$. Note that $\hat x_-=F^{N_u}\circ P^{F^{-N_u(\Gamma)}}_- \circ \Phi^{-t_u}(x_-)$ and $\hat x_+=F^{-N_s}\circ P^{F^{N_s(\Gamma)}}_+ \circ \Phi^{t_s}(x_+)$.

Thus
\begin{eqnarray*}
S^{\Gamma_\Sigma}(\hat x_-)&=&\hat x_+\\&=&F^{-N_s}\circ P^{F^{N_s(\Gamma)}}_+ \circ \Phi^{t_s}(x_+) \\
&=&F^{-N_s}\circ P^{F^{N_s(\Gamma)}}_+ \circ \Phi^{t_s}\circ S(x_-)\\
&=&F^{-N_s}\circ P^{F^{N_s(\Gamma)}}_+ \circ \Phi^{t_s}\circ S \circ \Phi^{t_u}\circ (P_-^{F^{-N_u}(\Gamma)})^{-1}\circ F^{-N_u} (\hat x_-).
\end{eqnarray*}

Hence
\begin{eqnarray*}
F^{N_s}\circ S^{\Gamma_\Sigma}\circ F^{N_u}= P^{F^{N_s(\Gamma)}}_+ \circ \Phi^{t_s}\circ S^\Gamma \circ \Phi^{t_u}\circ (P_-^{F^{-N_u}(\Gamma)})^{-1}.\end{eqnarray*}

The conclusion of the proposition now follows from the definition of the transition map in the flow case and  the definition of the transition map in the map case.
\end{proof}

\section{Topological method of correctly aligned windows}\label{section:topological}

We review briefly the topological method of correctly aligned windows. We follow \cite{GideaZ04a}. See also \cite{GideaR03,GideaL06}.

\begin{definition}
An $(m_1,m_2)$-window in an $m$-dimensional manifold $M$, where $m_1+m_2=m$, is a
compact subset $R$ of $M$ together with a $C^0$-parametrization given by a
homeomorphism $\chi$ from some open neighborhood of $[0,1]^{m_1}\times [0,1]^{m_2}$ in $\mathbb{R}^{m_1}\times \mathbb{R}^{m_2}$ to an open subset of $M$, with $R=\chi([0,1]^{m_1}\times [0,1]^{m_2})$, and with a choice of an `exit set' \[R^{\rm
exit} =\chi \left(\partial[0,1]^{m_1}\times [0,1]^{m_2} \right )\]
and  of an `entry set'  \[R^{\rm entry}
=\chi \left([0,1]^{m_1}\times
\partial[0,1]^{m_2}\right ).\]
\end{definition}
We adopt the following notation:
$R_\chi=\chi^{-1}(R)$, $(R^{\rm exit}
)_\chi=\chi^{-1}(R^{\rm exit} )$, and $(R^{\rm
entry})_\chi=\chi^{-1}(R^{\rm entry})$.
(Note that
$R_\chi=[0,1]^{m_1}\times[0,1]^{m_2}$, $(R^{\rm
exit} )_\chi=\partial [0,1]^{m_1}\times[0,1]^{m_2}$, and $(R^{\rm entry}
)_\chi=[0,1]^{m_1}\times \partial [0,1]^{m_2}$.)
When the local parametrization
$\chi$ is evident from context, we suppress the subscript $\chi$
from the notation.

\begin{definition}\label{defn:corr}
Let  $R_1$ and $R_2$ be $(m_1,m_2)$-windows, and let $\chi_1$ and $\chi_2$ be the corresponding local parametrizations.
Let $F$ be a continuous
map on $M$ with $F(\textrm {im}(\chi_1))\subseteq \textrm
{im}(\chi_2)$.  We say that $R_1$ is correctly aligned with
$R_2$ under $F$ if the following conditions are satisfied:
\begin{itemize}
\item[(i)] There exists a continuous homotopy $h:[0,1]\times
(R_1){\chi_1} \to {\mathbb R}^{m_1} \times {\mathbb R}^{m_2}$,
   such that the following conditions hold true
   \begin{eqnarray*}
      h_0&=&F_\chi, \\
      h([0,1],(R^{\rm exit}_1)_{\chi_1}) \cap (R_2)_{\chi_2} &=& \emptyset, \\
      h([0,1],(R_1)_{\chi_1}) \cap (R_2^{\rm entry})_{\chi_2} &=& \emptyset,
   \end{eqnarray*}
   where $F_\chi= \chi_2^{-1}\circ F\circ \chi_1$,  and
\item[(ii)]
the map $A_{y_0}:\mathbb{R}^{m_1}\to\mathbb{R}^{m_1}$ defined by
$A_{y_0}(x)=\pi _{m_1}\left(h_{1}(x, y_0)\right )$ satisfies
\begin{eqnarray*}
A_{y_0}\left ( \partial[0,1]^{m_1}\right )\subseteq \mathbb
{R}^{m_1}\setminus [0,1]^{m_1},\\\deg({A_{y_0}},0)\neq 0,\end{eqnarray*}
where $\pi_{m_1}: \mathbb{R}^{m_1}\times \mathbb{R}^{m_2}\to
\mathbb{R}^{m_1}$ is the  projection onto the first
component, and $\deg$ is the Brouwer degree of the map $A_{y_0}$ at
$0$.
\end{itemize}
\end{definition}

The following result  allows the detection of orbits with prescribed itineraries.
\begin{theorem}
%[Existence of orbits with prescribed trajectories]
\label{theorem:detorb} Let $R_i$ be a collection of
$(m_1,m_2)$-windows in $M$, where $i\in\mathbb{Z}$ or
$i\in\{0,\ldots, d-1\}$, with $d>0$ (in the latter case, for
convenience, we let \hbox{$R_{i}=R_{(i\,{\rm mod}\, d)}$} for all
$i\in \mathbb{Z}$). Let $F_i$ be a collection of continuous maps on
$M$. If $R_i$ is correctly aligned with $R_{i+1}$, for all $i$, then
there exists a point $p\in R_0$ such that
\[(F_{i}\circ \ldots\circ F_{0})(p)\in R_{i+1},\]
Moreover, if $R_{i+k}=R_{i}$ for some $k>0$ and all $i$, then the
point $p$ can be chosen periodic in the sense
\[(F_{k-1}\circ \ldots\circ F_{0})(p)=p.\]
\end{theorem}
Often,  the maps $F_i$ represent
different powers of the return map associated to a certain
surface of section. The orbit of the point $p$ found above is not necessarily unique.

The correct alignment of windows is robust, in the sense that if two
windows are correctly aligned under a map, then they remain
correctly aligned under a sufficiently small perturbation of the
map.

\begin{proposition}\label{prop:windowsrobust}
Assume $R_1,R_2$ are $(m_1,m_2)$-windows in $M$.
Let $G$ be a continuous maps on $M$. Assume that $R_1$ is
correctly aligned with $R_2$ under $G$. Then there exists $\epsilon
>0$, depending on the windows $R_1,R_2$ and $G$,  such that, for every continuous map $F$ on $M$ with  $\|F(x) - G(x)\| < \epsilon$ for all $x \in
R_1$, we have that $R_1$ is correctly aligned with $R_2$ under $F$.
\end{proposition}

Also, the correct alignment satisfies a natural product property.
Given two windows and a map, if each window can be written as a
product of window components, and if the components of the first
window are correctly aligned with the corresponding components of
the second window under the appropriate components of the map, then
the first window is correctly aligned with the second window under
the given map. For example, if we consider a pair of windows in a neighborhood of a
normally invariant normally hyperbolic invariant manifold, if the center components
of the windows are correctly aligned and the hyperbolic components of the windows
are also correctly aligned, then the windows are correctly aligned.
Although the product property is quite intuitive, its
rigorous statement is rather technical, so we will omit it here.
The details can be found in \cite{GideaL06}.

\section{A shadowing lemma for normally hyperbolic invariant manifolds}\label{sec:shadowing}

In this section we present a shadowing lemma-type of result.
It is assumed the existence of a sequence of windows in the normally hyperbolic
invariant manifold $\Lambda$,  with the windows of the same dimension $l$ as $\Lambda$. It is assumed
that  the sequence of windows is made up of pairs of windows $D_i^-,D^+_{i+1}$ that are  correctly aligned under the transition map $S^\Gamma_{N^-_i, N^+_{i+1}}$, alternately with pairs of windows  $D^+_{i+1},D^-_{i+1}$ that are correctly aligned under some power $F^{N_i^0}$ of the restriction of $F$ to $\Lambda$. Here, the superscript $\pm$ for the windows $D^\pm_{i}$ suggest that $D^\pm_{i}$ is typically obtained by taking some positive (negative) iteration of some other window that lies in the codomain (domain) of the scattering map.

It is required that the numbers $N^0_i,N^-_i,N^+_i$ can be chosen arbitrarily large, but uniformly bounded relative to $j$.
The conclusion is that there exists a true
orbit in the full space dynamics that follows these windows arbitrarily closely. The resulting orbit is not necessarily unique.

The result below provides a method to reduce the
problem of the existence of orbits in the full dimensional phase
space to a lower dimensional problem of the existence of
pseudo-orbits in the normally hyperbolic invariant manifold.

\begin{theorem}\label{lem:shadowing1}
Assume that there exists a  bi-infinite
sequence $\{D^+_i,D^-_i\}_{i\in\mathbb{Z}}$ of $l$-dimensional windows contained in a compact subset of $\Lambda$ such that,  for any integers   $n^0_1,n^-_1,n^+_1>0$,  there exist  integers $n^0_2>n^0_1$, $n^-_2>n^-_1$, $n^+_2>n^+_1$  and
 sequences of integers $\{N^0_i,N^-_i,N^+_i,\}_{i\in\mathbb{Z}}$  with $n^0_1<N^0_i<n^0_2$, $n^-_1<N^-_i<n^-_2$, $n^+_1<N^+_i<n^+_2$
such that the following properties hold for all $i\in\mathbb{Z}$:
\begin{itemize}
\item[(i)] $F^{-N^+_{i}}(D^+_{i})\subseteq U_+$ and $F^{N^-_i}(D^-_{i})\subseteq U_-$.
\item[(ii)] $D^-_{i}$ is
correctly aligned with $D^+_{i+1}$ under the transition map
$S^\Gamma_{N^-_i,N^+_{i+1}}=F^{N^+_{i+1}}\circ S\circ F^{N^-_{i}}$.\item[(iii)] $D^+_{i}$ is correctly aligned with
$D^-_{i}$ under the iterate $F^{N^0_{i}}$ of
$F_{\mid \Lambda}$.
\end{itemize}
Then, for every $\eps>0$,    there exist an orbit  $\{F^{N}(z)\}_{N\in\mathbb{Z}}$ of $F$ for some $z\in M$, and
an increasing sequence of integers $\{N_i\}_{i\in\mathbb{Z}}$  with $N_{i+1}=N_i+N^+_{i+1}+N^0_{i+1}+N^-_{i+1}$ such
that, for all $i$:
\[ \begin{split}
d(F^{N_i}(z),\Gamma)<\eps,\\
d(F^{N_{i}-N^-_{i}}(z), D^-_{i})<\eps,\\
d(F^{N_i+N^+_{i+1}}(z), D^+_{i+1})<\eps.
\end{split}\]

\end{theorem}

\begin{proof}
The idea of this proof is to `thicken' the windows $D^+_{i},D^-_{i}$ in
$\Lambda$ to  full-dimensional windows $R^-_i,R^+_i$ in $M$, so that the
successive windows in the sequence $\{R^-_i,R^+_i\}_i$  are correctly
aligned under some appropriate iterations of the map $F$.
The argument is done in several steps.  In  the first three steps, we only specify the relative sizes of the windows involved in each step. In the fourth step, we explain how to make the choices of the sizes of the windows uniform.

\textit{Step 1.} At this step, we  take a pair of $l$-dimensional windows $D^-_{i}$ and $D^+_{i+1}$ as in the statement of the theorem and, through applying iteration and the wave maps  we construct two  $(l+2n)$-dimensional windows  $\bar R^-_{i}$ and $\bar R^+_{i+1}$ near $\Gamma$ such that $\bar R^-_{i}$ is correctly aligned with $\bar R^+_{i+1}$ under the identity map.

Note that conditions (i) and (ii) imply that
$\hat D^-_{i}:=F^{N^-_{i}}(D^-_{i})\subseteq U_-\subseteq \Lambda$ is
correctly aligned with $\hat D^+_{i+1}:=F^{-N^+_{i+1}}(D^+_{i+1})\subseteq U_+\subseteq \Lambda$ under
the scattering map $S$.
Let $\bar D^-_{i}=(\Omega_-^\Gamma)^{-1}(\hat D^-_{i})$ and
$\bar D^+_{i+1}=(\Omega_+^\Gamma)^{-1}(\hat D^+_{i+1})$ be the copies of $\hat D^-_{i}$ and
$\hat D^+_{i+1}$, respectively, in the homoclinic channel $\Gamma$.  By
making some arbitrarily small changes in the sizes of their exit and entry directions, we
can alter the windows $\hat D^-_{i}$ and $\hat D^+_{i+1}$ such that $\hat D^-_{i}$ is correctly aligned with $\bar D^-_{i}$
under $(\Omega^-_\Gamma)^{-1}$, $\bar D^-_{i}$ is correctly aligned with $\bar D^+_{i+1}$ under the identity mapping, and
$\bar D^+_{i+1}$ is correctly aligned with $\hat D^+_{i+1}$ under
$\Omega^+_\Gamma$.

We `thicken' the $l$-dimensional windows $\bar D^-_{i}$ and $\bar D^+_{i+1}$
in $\Gamma$, which are correctly aligned under the identity mapping,
to $(l+2n)$-dimensional windows that are correctly aligned under the
identity map. We now explain the `thickening' procedure.

First, we describe how to thicken $\bar D^-_i$ to a full dimensional window $\bar R^-_i$.
We choose some $0<\bar\delta^-_{i}<\eps$ and $0<\bar\eta^-_{i}<\eps$. At each
point $x\in \bar D^-_i$ we choose an $n$-dimensional closed ball $\bar
B^-_{\bar\delta^-_{i}}(x)$ of radius $\bar\delta^-_{i}$ centered at $x$ and
contained in $W^u({x_-})$, where $x_-=\Omega_{-}^\Gamma(x)$. We take the
union $\bar\Delta^-_{i}:=\bigcup_{x\in \bar D^-_{i}}\bar
B^{u}_{\bar\delta^-_{i}}(x)$. Note that $\bar\Delta^-_{i}$ is contained in
$W^u(\Lambda)$ and is homeomorphic to an $(l+n)$-dimensional
rectangle.  We define the exit set and the entry set of this
rectangle as follows:
\[\begin{split}(\bar\Delta^-_{i})^{\rm exit}:=\bigcup_{x\in (\bar D^-_{i})^{\rm exit}}
 \bar B^{u}_{\bar\delta^-_{i}}(x) \cup
\bigcup_{x\in \bar D^-_{i}}\partial \bar B^{u}_{\bar\delta^-_{i}}(x),\\
(\bar\Delta^-_{i})^{\rm entry}:=\bigcup_{x\in (\bar D^-_{i})^{\rm entry}}
 \bar B^{u}_{\bar\delta^-_{i}}(x).
\end{split}\]

We consider the normal bundle $N^+$ to $W^u(\Lambda)$. At
each point $y\in \bar\Delta^-_{i}$,
we choose an $n$-dimensional closed ball $\bar B^+_{\bar\eta^-_{i}}(y)$
centered at $y$ and contained in the image of  $N^+_y\subseteq T_yM$ under the
exponential map $\exp_y:N^+_y\to M$. We let
$\bar R^-_{i}:=\bigcup_{y\in\bar\Delta^-_{i}}\bar B^{s}_{\bar\eta^-_{i}}(y)$.
By the Tubular Neighborhood Theorem (see, for example
\cite{BurnsG05}), we have that for $\bar\eta^-_{i}>0$ sufficiently
small, the set $\bar R^-_{i}$ is a homeomorphic copy of an
$(l+2n)$-rectangle. We now define the exit set and the
entry set of $\bar R^-_{i}$ as follows:
\[\begin{split} (\bar R^-_{i})^{\rm
exit}:=\bigcup_{y\in(\bar\Delta^-_{i})^{\rm
exit}}\bar B^{s}_{\bar\eta^-_{i}}(y),\\
(\bar R^-_{i})^{\rm entry}:=\bigcup_{y\in(\bar\Delta^-_{i})^{\rm
entry}}\bar B^{s}_{\bar\eta^-_{i}}(y) \cup
\bigcup_{y\in(\bar\Delta^-_{i})}\partial \bar B^{s}_{\bar\eta^-_{i}}(y).
\end{split}
\]

Second, we describe in a similar fashion how to thicken $\bar D^+_{i+1}$ to a full dimensional window
$\bar R^+_{i+1}$. We choose $0<\bar\delta^+_{i+1}<\eps$ and
$0<\bar\eta^+_{i+1}<\eps$. We consider the $(l+n)$-dimensional rectangle
$\bar\Delta^+_{i+1}:=\bigcup_{x\in \bar D^+_{i+1}}\bar
B^{s}_{\bar\eta^+_{i+1}}(x)\subseteq W^s(\Lambda)$, where $\bar
B^+_{\bar\eta^+_{i+1}}(x)$ is the $n$-dimensional closed ball of radius
$\bar\eta^+_{i+1}$ centered at $x$ and contained in $W^s(x_+)$, with
$x_+=\Omega_{+}^\Gamma(x)$. The exit set and entry set of this window
are defined as follows:
\[\begin{split}(\bar\Delta^+_{i+1})^{\rm exit}:=\bigcup_{x\in (\bar D^+_{i+1})^{\rm exit}}
 \bar B^{s}_{\bar\eta^+_{i+1}}(x),\\
(\bar\Delta^+_{i+1})^{\rm entry}:=\bigcup_{x\in (\bar D^+_{i+1})^{\rm
entry}}
 \bar B^{s}_{\bar\eta^+_{i+1}}(x)\cup \bigcup_{x\in (\bar D^+_{i+1})}
\partial \bar B^{s}_{\bar\eta^+_{i+1}}(x).
\end{split}\] We let
$\bar R^+_{i+1}:=\bigcup_{y\in\bar\Delta^+_{i+1}}\bar B^{u}_{\bar\delta^+_{i+1}}(y)$,
where $\bar B^-_{\bar\delta^+_{i+1}}(y)$ is the  $n$-dimensional closed ball
centered at $y$ and contained in the image of  $N^-_y\subseteq T_yM$ under the
exponential map $\exp_y:N^-_y\to M$, and $N^-$ is the normal bundle to $W^s(\Lambda)$.
The Tubular
Neighborhood Theorem implies that for $\bar\delta^+_{i+1}>0$
sufficiently small the set $\bar R^+_{i+1}$ is a homeomorphic copy of a
$(l+2n)$-rectangle. The exit set and the entry set of
$\bar R^+_{i+1}$ are defined by:
\[\begin{split} (\bar R^+_{i+1})^{\rm
exit}:=\bigcup_{y\in(\bar\Delta^+_{i+1})^{\rm
exit}}\bar B^{u}_{\bar\delta^+_{i+1}}(y)\cup
\bigcup_{y\in(\bar\Delta^+_{i+1}}\partial \bar B^{u}_{\bar\delta^+_{i+1}}(y),\\
(R^+_{i+1})^{\rm entry}:=\bigcup_{y\in(\bar\Delta^+_{i+1})^{\rm
entry}}\bar B^{u}_{\bar\delta^+_{i+1}}(y).
\end{split}
\]

This completes the description of the thickening of the
$l$-dimensional window $\bar D^-_{i}$ into an $(l+2n)$-dimensional window
$\bar R^-_{i}$, and of the thickening of the $l$-dimensional window
$\bar D^+_{i+1}$ into an $(l+2n)$-dimensional window $\bar R^+_{i+1}$. Note that,
by construction, $\bar R^-_{i}$ and $\bar R^+_{i+1}$ are both contained in an
$\eps$-neighborhood of $\Gamma$.

Now we want to make $\bar R^-_{i}$ correctly aligned with $\bar R^+_{i+1}$
under the identity map. This is achieved by choosing
$\bar\delta^+_{i+1}$ sufficiently small relative to $\bar\delta^-_{i}$, and by choosing $\bar\eta^-_{i}$ sufficiently small relative to $\bar\eta^+_{i+1}$. Thus, we have $\bar\delta^-_{i}>\bar\delta^+_{i+1}$ and
$\bar\eta^-_{i}<\bar\eta^+_{i+1}$ (we stress that these inequalities alone may not suffice for  the correct alignment). Choosing $\bar\delta^+_{i+1}$ and $\bar\eta^-_{i}$ small enough agrees with the constraints imposed by the Tubular Neighborhood Theorem.

\textit{Step 2.}
At this step, we expand the given $l$-dimensional window $D^-_i$ to an $(l+2n)$-dimensional window $R_i^-$ such that  $R_i^-$ is correctly aligned with $\bar R_i^-$ under some positive iterate, and we also expand the given $l$-dimensional window $D^+_{i+1}$ to an $(l+2n)$-dimensional window $R_{i+1}^+$ such that  $\bar R_{i+1}^+$ is correctly aligned with $\bar R_{i+1}^+$ under some positive iterate.

We take a negative iterate $F^{-M}(\bar R^-_{i})$ of $\bar R^-_{i}$, where $M>0$.
We have that $F^{-M}(\Gamma)$ is $\eps$-close to $\Lambda$ on a
neighborhood in the $C^1$-topology,  for all $M$ sufficiently large. The vectors tangent to the fibers $W^u(x_-)$ in $\bar R^-_i$ are contracted, and the vectors transverse to
$W^u(\Lambda)$ along $\bar R^-_i\cap W^u(\Lambda)$ are expanded by the
derivative of $F^{-M}$. We choose and fix $M=N^-_i$ sufficiently large. We obtain that, in particular,
$F^{-N^-_i}(\bar R^-_{i})$ is  $\eps$-close to
$D^-_i=F^{-N^-_{i}}(\hat D^-_{i})$.

We now construct a window $R^-_{i}$ about $D^-_{i}$ that is correctly
aligned with the window $F^{-N^-_{i}}(\bar R^-_{i})$ under the identity. Note that each closed ball $\bar B^{u}_{\delta^-_{i}}(x)$, which is a part of  $\bar\Delta^- _{i}$, gets exponentially contracted  as it is mapped into  $W^u(F^{-{N^-_i}}(x_-))$ by $F^{-{N^-}_{i}}$. By the Lambda Lemma (Proposition \ref{thm:lambdalemma}), each closed ball $\bar B^s_{\eta^-_{i}}(y)$ with $y\in \bar\Delta^-_{i}$, which is a part of $\bar R^-_{i}$,  $C^1$-approaches a subset of $W^s(F^{-{M}}(y_-))$ under $F^{-{M}}$, as $M\to\infty$. For $N^-_i$ sufficiently large, we may assume that $F^{-{N^-_i}}(\bar B^s_{\eta^-_{i}}(y))$  is $\eps$-close to a subset of $W^s(F^{-{N^-_i}}(y^-))$ in the $C^1$-topology, for all $y\in \bar\Delta^- _{i}$.
As $\hat D^-_{i}$ is correctly aligned with $\bar D^-_{i}$ under
$(\Omega_-^\Gamma)^{-1}$, we have that
$D^-_{i}=F^{-N^-_{i}}(\hat D^-_{i})$ is correctly aligned with
$F^{-N^-_{i}}(\bar D^-_{i})$ under
$(\Omega_-^{F^{-N^-_{i}}(\Gamma)})^{-1}$. In other words,
$D^-_{i}$ is
correctly aligned under the identity mapping with the projection of
$F^{-N^-_{i}}(\bar D^-_{i})$ onto $\Lambda$ along the unstable fibres.
Let us consider $0<\delta^-_{i}<\eps$ and $0<\eta^-_{i}<\eps$.

To define the window $R^-_{i}$  we use a local linearization of the normally hyperbolic invariant manifold.

For $\Lambda$ normally hyperbolic, let $N\Lambda=(E^u\oplus E^s)_{\mid\Lambda}=\bigcup_{p\in \Lambda}\{p\}\times E^u_p\times E^s_p$ be the normal bundle to $\Lambda$, and $NF=TF_{\mid N\Lambda}$, where \[TF(p,v^u,v^s)=(F(p),DF_p(v^u),DF_p(v^s))\textrm{ for all } p\in\Lambda,v^u\in E^u, v^s\in E^s.\] By Theorem 1 in \cite{PughS70}, there exists a homeomorphisms $h$ from an open neighborhood of the zero section of $N\Lambda$ to
a neighborhood of $\Lambda$ in $M$ such that $F\circ h=h\circ NF$.

Since  $D^-_i$ is contractible the bundles are trivial on
$D^-_i$ and we can identify $(E^u\oplus E^s)_{D^-_i}$ with
$D^-_i \times E^u_{x } \times E^s_{x }$. At each point $x\in D^-_i$ we define  a rectangle $H^-_i(x)$ of the type $h(\{x\}\times\bar B^u_{\delta^-_{i}}(0)\times \bar B^s_{\eta^-_{i}}(0))$, where $\bar B^u_{\delta^-_{i}}(0)$ is the closed ball centered at $0$ of radius $\delta^-_{i}$  in the unstable space $E^u_x$, and $\bar B^s_{\eta^-_{i}}$ is the closed ball centered at $0$ of radius $\eta^-_{i}$ in the stable space $E^s_x$. We set the exit and entry sets of $H^-_{i}(x)$ as $(H^-_{i}(x))^{\rm exit}=h(\{x\}\times\partial \bar B^u_{\delta^-_{i}}(0)\times \bar B^s_{\eta^-_{i}}(0))$ and $(H^-_{i}(x))^{\rm entry}=h(\{x\}\times\bar B^u_{\delta^-_{i}}(0)\times \partial \bar B^s_{\eta^-_{i}}(0))$.

Then we define the window $R^-_{i}$  as follows:
\[\begin{split}R^-_{i}=\bigcup_{x\in D^-_{i}}H^-_{i}(x),\\
(R^-_{i})^{\rm exit}=\bigcup_{x\in (D^-_{i})^{\rm exit}}H^-_{i}(x) \cup \bigcup_{x\in D^-_{i}}(H^-_{i}(x))^{\rm exit},\\
(R^-_{i})^{\rm entry}=\bigcup_{x\in (D^-_{i})^{\rm entry}}H^-_{i}(x) \cup \bigcup_{x\in D^-_{i}}(H^-_{i}(x))^{\rm entry}.
\end{split}\]

In order to ensure the correct alignment of $R^-_{i}$ with $F^{-N^-_i}(\bar R^-_{i})$ under the identity map, it is sufficient to choose $\delta^-_i,\eta^-_i$ such that
$\bigcup_{x\in D^-_{i}}h(\{x\}\times\bar B^u_{\delta^-_{i}}(0)\times \{0\})$ is correctly aligned with $F^{-N^-_{i}}(\bar\Delta^-_{i})$ under the identity map (the exit sets of both windows being in the unstable directions),
and that each closed ball $F^{-N^-_{i}}(\bar B^s_{\eta^-_{i}})$ intersects $R_{i}$ in a closed ball  that is contained in the interior of $F^{-N^-_{i}}(\bar B^s_{\eta^-_{i}})$. The existence of suitable $\delta^-_i,\eta^-_i$ follows from the
exponential contraction of  $\bar\Delta^-_{i}$ under negative iteration, and from the Lambda Lemma applied to  $\bar B^s_{\eta^-_{i}}(y)$ under negative iteration.

In a similar fashion, we construct a window $R^+_{i+1}$   contained in an $\eps$-neighborhood of $\Lambda$ such  that $\bar R^+_{i+1}$ is correctly aligned with $R^+_{i+1}$ under $F^{N^+_{i+1}}$.
The window $R^+_{i+1}$, and its entry and exit sets, are  defined by:
\[\begin{split}R^+_{i+1}=\bigcup_{x\in D^+_{i+1}}H^+_{i+1}(x),\\
(R^+_{i+1})^{\rm exit}=\bigcup_{x\in (D^+_{i+1})^{\rm exit}}H^+_{i+1}(x) \cup \bigcup_{x\in D^+_{i+1}}(H^+_{i+1}(x))^{\rm exit},\\
(R^+_{i+1})^{\rm entry}=\bigcup_{x\in (D^+_{i+1})^{\rm entry}}H^+_{i+1}(x) \cup \bigcup_{x\in D^+_{i+1}}(H^+_{i+1}(x))^{\rm entry},
\end{split}\]
where $H^+_{i+1}(x)=h(\{x\}\times\bar B^u_{\delta^+_{i+1}}(0)\times \bar B^s_{\eta^+_{i+1}}(0))$, $(H^+_{i+1}(x))^{\rm exit}$, and $(H^+_{i+1}(x))^{\rm entry}$  are defined as before for some appropriate choices of radii $\delta^+_{i+1},\eta^+_{i+1}>0$.

\textit{Step 3.}
At this step, we take the $(l+2n)$-dimensional window $R_{i+1}^+$ and $\bar R_{i+1}^-$ as constructed in the previous step, and we make
$R_{i+1}^+$ correctly aligned with $\bar R_{i+1}^-$
under some positive iterate.

Suppose that we have constructed the window $R^+_{i+1}$ about the $l$-dimensional rectangle $D^+_{i+1}\subseteq \Lambda$ and the window $R^-_{i+1}$ about the $l$-dimensional rectangle $R^-_{i+1}\subseteq \Lambda$. Under positive iterations,  the rectangle $\bar B^u_{\delta^+_{i+1}}(0)\times \bar B^s_{\eta^+_{i+1}}(0)\subseteq E^u\oplus E^s$ gets exponentially expanded in the unstable direction and exponentially contracted in the stable direction by $DF$. Thus $\bar B^u_{\delta^+_{i+1}}(0)\times \bar B^s_{\eta^+_{i+1}}(0)$ is correctly aligned with $\bar B^u_{\delta^-_{i+1}}(0)\times \bar B^s_{\eta^-_{i+1}}(0)$ under  the  power $DF^{N^0_{i+1}}$ of $DF$, provided $N^0_{i+1}$ is sufficiently large.
This implies that  $F^{N^0_{i+1}}(h(\{x\}\times \bar B^u_{\delta^+_{i+1}}(0)\times \bar B^s_{\delta^+_{i+1}}(0)))$ is correctly aligned with $h({F^{N^0_{i+1}}(x)}\times \bar B^u_{\delta^-_{i+1}}(0)\times \bar B^s_{\delta^-_{i+1}}(0))$ under the identity map (both rectangles are contained in $h({F^{N^0_{i+1}}(x)}\times E^u\times E^s)$).

Since $D^+_{i+1}$ is correctly aligned with $D^-_{i+1}$ under $F^{N^0_i}$, the product property of correctly aligned windows implies that $R^+_{i+1}$ is correctly aligned with $R^-_{i+1}$ under $F^{N^0_{i+1}}$, provided that $N^0_{i+1}$ is sufficiently large.

\textit{Step 4.}
At this step we will use the previous steps to construct a bi-infinite sequences of windows $\{R^\pm_{i},\bar R^\pm_{i}\}_{i\in\mathbb{Z}}$ such that, for each $i$, the windows $\{R^\pm_{i}\}$ are obtained by thickening the rectangles $\{D^\pm_i\}\subseteq \Lambda$, the windows $\{\bar R^\pm_{i+1}\}$  are obtained by thickening some rectangles  $\{\bar D^\pm_i\}\subseteq\Gamma$, and, moreover,   $R^-_{i}$ is correctly aligned with $\bar R^-_{i}$ under $F^{N^-_{i}}$, $\bar R^-_{i}$  is correctly aligned with $\bar R^+_{i+1}$  under the identity map, $\bar R^+_{i+1}$ is correctly aligned with $R^+_{i+1}$ under $F^{N^+_{i+1}}$, and $ R^+_{i+1}$ is correctly aligned with $R^-_{i+1}$ under $F^{N^0_i}$.

We can assume without loss of generality that $\Lambda$ and $\Gamma$  are compact. We fix an $\eps$-neighborhood $V$ of $\Lambda$.  Using the compactness of $\Lambda$ and $\Gamma$ and the uniform boundedness of the iterates $N^-_i,N^+_i,N^0_i$, we now show how to choose  the sizes of the stable and unstable components of the windows $\{R^\pm_{i},\bar R^\pm_{i}\}_{i\in\mathbb{Z}}$ constructed in the previous steps in a uniformly bounded manner.

For each point $x$ in $\Lambda$ we consider a $(2n)$-dimensional window
$h(\{x\}\times \bar B^u_\delta\times\bar B^s_\eta )$, for some $0<\delta,\eta<\eps$, where $h$ is the local conjugacy between $F$ and $DF$ near $\Lambda$. Then $F^{N^0_i}(h(\{x\}\times \bar B^u_\delta\times\bar B^s_\eta )$ is correctly aligned with  $h(\{F^{N^0_i}(x)\}\times \bar B^u_\delta(0)\times\bar B^s_\eta(0))$, for all $n^0_1\leq N^0_i\leq n^0_2$, provided that $n^0_1$ is chosen sufficiently large. For each $i$, we  thicken $D^+_{i}$ and $D^-_{i}$ into full dimensional windows $R^+_{i}$ and of $R^-_{i}$ respectively, as described in Step 2, where for the sizes of the components of these windows we choose $\delta^\pm_i=\delta$ and $\eta^\pm_i=\eta$ for all $i$. Since $D^+_{i}$ is correctly aligned with $D^-_{i}$ under $F^{N^0_{i}}$, then, as in Step 3, it follows  that $R^+_{i}$ is correctly aligned with $R^-_{i}$ under $F^{N^0_{i}}$.

We also define the set \[\Upsilon^0=\bigcup_{x\in\Lambda} h(\{x\}\times \bar B^s_\eta(0)\times \bar B^s_\delta(0)).\]
This set cannot be realized as a window since it does not have exit/entry directions associated to the $\Lambda$ components. However, for each $x\in\Lambda$, the set $h(\{x\}\times \bar B^s_\eta(0)\times \bar B^u_\delta(0))$ is a well defined window, with the exit given by the hyperbolic unstable directions.   Note that $\Upsilon^0(x)\subseteq h(\{x\}\times \bar W^u(x)\times \bar W^s(x))$ for each $x\in\Gamma$.

We let  $\bar\Delta^-=\bigcup _{x\in \Gamma}\bar B^u_{\bar\delta^-}(x)$, with $\bar B^u_{\bar\delta^-}(x)$ being the closed  ball centered at $x$ of radius $\bar\delta^-$ in $W^u(x_-)$. For each point $y\in \bar\Delta^-$ we consider the closed ball $\bar B^s_{\bar\eta^-}(y)$ centered at $y$ of radius $\bar\eta^-$ in the image under $\exp_y$ of the normal subspace $N_y$ to  $W^u(\Lambda)$ at $y$. Similarly, we let $\bar\Delta^+=\bigcup _{x\in \Gamma}\bar B^s_{\bar\eta^+}(x)$, where $\bar B^u_{\bar\eta^+}(x)\subseteq W^s(x_+)$, and for each  $y\in \bar\Delta^+$ we consider the closed ball $\bar B^u_{\bar\delta^+}(y)$  in the image under $\exp_y$ of the normal subspace $N_y$ to  $W^s(\Lambda)$ at $y$.
We define the sets
\[
\Upsilon^-=\bigcup_{y\in \bar\Delta^-} \bar B^s_{\bar\eta^-}(y),\,
\Upsilon^+=\bigcup_{y\in \bar\Delta^+} \bar B^u_{\bar\delta^+}(y).
\]
These sets cannot be realized as windows as there are no well defined exit/entry directions associated to their $\Gamma$ components. However, for each $x\in \Gamma$, the set $\Upsilon^-(x)= \bigcup_{y\in B^u_{\bar\delta^-}(x)} \bar B^s_{\bar\eta^-}(y)$ is a well defined $(2n)$-dimensional window, with the exit  given by the hyperbolic unstable directions.    Note that $\Upsilon^-(x)\subseteq \bigcup_{y\in W^u(x_-)}\exp_y(N_y)$. The intersection of $\bigcup_{y\in W^u(x_-)}\exp_y(N_y)$ with $\Upsilon^+$ defines a window $\Upsilon^+(x)$ with the exit given by the hyperbolic unstable directions.
Due to the compactness of $\Gamma$, there exist $\delta^\pm,\eta^\pm$ such that $\Upsilon^-(x)$ is correctly aligned with $\Upsilon^+(x)$ for all $x\in\Gamma$. We choose and fix such  $\delta^\pm,\eta^\pm$.
We define the windows $\bar R^\pm_i$ at Step 1 with the choices of $\delta^\pm_i=\delta^\pm$, and $\eta^\pm_i=\eta^\pm$, for all $i$. It follows that $\bar R^-_i$ is correctly aligned with $\bar R^+_i$ under the identity map for all $i$.

Due to the compactness of $\Gamma$ and the uniform expansion and contraction of the hyperbolic directions, there  exist $n^-_1,n^+_1$ such that, for all $N^-_i>n^-_1$, $N^+_i>n^+_2$, we have $F^{-N^-_i}(\Gamma)\subseteq V$ and
$F^{N^+_i}(\Gamma)\subseteq V$ for all $i\in\mathbb{Z}$, where $V$ is the neighborhood of $\Lambda$ where the local linearization is defined. For any such $n^-_1,n^+_1$, the assumptions of Lemma \ref{lem:shadowing1} provide us with some $n^-_2>n^-_i$, $n^+>n^-_i$. Moreover, we choose $n^-_1,n^+_1$  such that for all $N^-_i>n^-_1$, $N^+_i>n^+_2$ we have \begin{itemize}
\item[(i)] $\Upsilon^0(F^{-N^-_i}(x_-))$ is correctly aligned with $F^{-N^-_i}(\Upsilon^-)\cap  h(\{\tilde F^{-N^-_i}(x_-)\} \times\break \bar W^u(F^{-N^-_i}(x_-)) \times \bar W^s(F^{-N^-_i}(x_-)))]$ under the identity map, \item [(ii)] $F^{N^+_i}(\Upsilon^+(x))$ is correctly aligned with
$\Upsilon^0\cap h(\{\tilde F^{N^+_i}(x_+)\}\times \bar W^u(F^{N^+_i}(x_+)) \times \bar W^s(F^{N^+_i}(x_+)))]$  under the
identity map.
\end{itemize}
From these choices, it follows that the windows $R^-_i, R^+_i$ constructed in Step 2 satisfy that  $R^-_i$ is correctly aligned with $\bar R^-_i$ under $F^{N^-_i}$, and $\bar R^+_i$ is correctly aligned with $R^+_i$ under $F^{N^+_i}$.

This concludes the construction of windows $\{R^\pm_{i},\bar R^\pm_{i}\}_{i\in\mathbb{Z}}$ of uniform sizes,  such that $R^-_{i}$ is correctly aligned with $\bar R^-_{i}$ under $F^{N^-_{i}}$, $\bar R^-_{i}$  is correctly aligned with $\bar R^+_{i+1}$  under the identity map, $\bar R^+_{i+1}$ is correctly aligned with $R^+_{i+1}$ under $F^{N^+_{i+1}}$, and $ R^+_{i+1}$ is correctly aligned with $R^-_{i+1}$ under $F^{N^0_i}$. The windows $R^\pm_i$ are contained in $\eps$-neighborhoods of the given rectangles $D^\pm_i$, respectively, and the windows  $R^\pm_i$ are contained in $\eps$-neighborhoods of some rectangles $\bar D^\pm_i$, respectively.

By Theorem \ref{theorem:detorb}, there exits an orbit $F^{N}(z)$ that visits the windows $\{R^\pm_{i},\bar R^\pm_{i}\}_{i\in\mathbb{Z}}$  in the prescribed order. More precisely, if $F^{N_i}(z)$ is the corresponding point in $\bar R^-_i\cap R^+_{i+1}$, then  $F^{N_i+N^+_{i+1}}(z)$ is in $R^+_{i+1}$, $F^{N_i+N^+_{i+1}+N^0_{i+1}}(z)$ is in $ R^-_{i+1}$, and\break $F^{N_i+N^+_{i+1}+N^0_{i+1}+N^-_{i+1}}(z)$ is in $\bar R^-_{i+1}\cap  \bar R^+_{i+2}$,    for all $i$. This means that $N_{i+1}=N_i+N^+_{i+1}+N^0_{i+1}+N^-_{i+1}$ for all $i$. The existence of the shadowing orbit concludes the proof.
\end{proof}

\section*{Acknowledgement} Part of this work has been done while M.G. was visiting
the Centre de Recerca Matem\`atica, Barcelona, Spain, for whose
hospitality he is very grateful.

\section*{Appendix} In this section we discuss  the  application of the  techniques discussed in this paper to show the existence of unstable orbits in dynamical systems. We describe two models for Hamiltonian instability and we explain how the transition map and the shadowing lemma
are utilized.

\subsubsection*{The large gap problem of Arnold diffusion} The first model is related to the  Arnold diffusion problem for Hamiltonian systems \cite{Arnold64}. This problem conjectures that  generic Hamiltonian systems that are close to integrable possess trajectories the move `wildly' and `arbitrarily far'. The model, from \cite{DelshamsLS2006},  describes  a rotator and a pendulum with a small, periodic coupling. Consider the time-dependent Hamiltonian,
\begin{eqnarray}\label{eqn:hamiltonian}H_\mu(p, q, I,\phi,
t)&=&h_0(I)+P_{\pm}(p,q)+\mu h(p, q,I,\phi, t; \mu),
\end{eqnarray}
where $(p, q, I,\phi, t)\in(\mathbb{R}\times \mathbb{T}^1)^2\times\mathbb{T}^1$, and
assume:
\begin{itemize}
\item[(i)] $V$, $h_0$ and $h$ are uniformly $C^r$ with $r$ sufficiently large;
\item[(ii)] $P_{\pm}(p,q)=\pm(\frac{1}{2}p^2+V(q))$ where $V$ is periodic in $q$ of
period $2\pi$ and has a unique non-degenerate global maximum at $(0,0)$; thus $P_{\pm}(p,q)$ has a  homoclinic orbit $(p^0 (\sigma),q^0 (\sigma))$ to $(0,0)$, with  $\sigma\in\mathbb{R}$;
\item[(iii)] $h_0$ satisfies a uniform twist condition $\partial^2
h_0/\partial I^2>\theta$, for some $\theta>0$ and all $I$ in
some interval $(I^-, I^+)$, with $I^-<I^+$ independent of $\mu$;
\item[(iv)] $h$ is a trigonometric polynomial in $(\phi,t)$, periodic of period $2\pi$ in both $\phi,t$, \item[(v)] The Melnikov potential associated to $(p^0 (\sigma),q^0 (\sigma))$, given by
\begin{eqnarray*}
\mathcal{L}(I,\phi,t)=&\displaystyle-\int_{-\infty}^{\infty}&\left
[h(p^0(\sigma),q^0(\sigma),I,\phi+\omega(I)\sigma,t+\sigma;0)\right .\\
& &\left .-h(0,0,I,\phi+\omega(I)\sigma,t+\sigma;0)\right ]d\sigma;
\end{eqnarray*}
where $\omega(I)=(\partial h_0/\partial I)(I)$, satisfies  the following non-degeneracy conditions:
\begin{itemize}
\item[(v.a)] For each $I\in(I^-,I^+)$, and each
$(\phi, t)$ in some open set in $\mathbb{T}^1\times \mathbb{T}^1$, the
map
\[\tau\in\mathbb{R} \to
\mathcal{L}(I,\phi-\omega(I)\tau,t-\tau)\in\mathbb{R}\] has a non-degenerate
critical point $\tau^*$, which can be parameterized as
\[\tau^*=\tau^*(I,\phi,t);\]

\item[(v.b)] For each $(I,\phi, t)$ as above, the
function\[(I,\phi,t)\to
\frac{\partial{\mathcal{L}}}{\partial\phi}(I,\phi-\omega(I)\tau^*,t-\tau^*)
\] is non-constant, negative for  $P_{-}$,  and positive for $P_{+}$;
\item[(v.c)] The perturbation term $h$
satisfies  some additional non-degeneracy conditions described in hypothesis (H5) of Theorem 7 in \cite{DelshamsLS2006}.
\end{itemize}
\end{itemize}
The objective is to show that there exists $\mu_0>0$ such that, for each $0<\mu<\mu_0$, the Hamiltonian has a trajectory $x(t)$  such that $I(x(0))<I^-$ and $I(x(T))>I^+$ for some $T>0$.

In the sequel, we show that the arguments in \cite{DelshamsLS2006,GideaL06,DelshamsGLS2008,GideaR09} can be combined with the techniques developed in this paper to prove the existence of diffusing trajectories as above.

(1) We can make the Hamiltonian $H_\mu$ into an autonomous  Hamiltonian by adding an extra variable $A$ symplectically conjugate with the time $t$. The variable $A$ has no dynamical role. We then restrict to an energy level, which can be parametrized by the variables $(p, q,I,\phi, t)$ with $A$ an implicit function of these variables.
When $\mu=0$ the set $\tilde\Lambda_0=\{(p, q,I,\phi, t),\, p=q=0, I\in[I^-,I^+]\}$ is a $3$-dimensional normally hyperbolic invariant manifold with boundary for the flow of $H_0$. It is foliated by $2$-dimensional invariant tori  $\tilde{\mathcal{T}}_{I'}=\{(p, q,I,\phi, t),\, p=q=0, I=I'\}$ with $I'\in[I^-,I^+]$.
The stable and unstable manifolds of $\tilde\Lambda_0$ coincide, i.e. $W^u(\tilde \Lambda_0)=W^s(\tilde \Lambda_0)$. Also $W^u(\tilde{\mathcal{T}}_{I'})=W^s(\tilde{\mathcal{T}}_{I'})$ for all $I'\in [I^-,I^+]$. This follows from (ii).

(2) There exists $\mu_0>0$ small such that for each $0<\mu<\mu_0$ the manifold $\tilde\Lambda_0$ has a continuation $\tilde\Lambda_\mu$ that is a  normally hyperbolic invariant manifold for the flow of $H_\mu$. Inside $\Lambda_\mu$ there are finitely many resonant regions, due to (iv).
Outside the resonant regions, the conditions (i), (iii)  allow  one to apply the KAM theorem, and obtain KAM tori that are at a distance of order $O(\mu^{3/2})$ from one another. The resonant regions yield  gaps of size $O(\mu^{j/2})$ between the KAM tori, where $j$ is the  order of the resonance. Only the resonances of order $1$ and $2$ are of interest, as they produce gaps of size $O(\mu)$ and $O(\mu^{1/2})$ respectively.
 Inside each resonant region,  there exist primary KAM tori and secondary KAM tori (homotopically trivial relative to $\Lambda_\eps$).
They can be chosen to be $O(\mu^{3/2})$ from one another.

(3) The conditions (v) on the Melnikov potential imply that if $\mu_0$ is chosen small
enough then $W^u(\tilde\Lambda_\mu)$ intersects $W^s(\tilde\Lambda_\mu)$ transversally at an angle $O(\mu)$. Choose and fix a homoclinic channel $\tilde \Gamma_\mu$ as in \eqref{eq:dynamical channel} and \eqref{eq:transverse foliation}, and let $S^{\tilde \Gamma_\mu}$ denote the corresponding scattering map for the flow of $H_\mu$, as in Definition \ref{defn:scattering_map_flow}. It turns out that the scattering map has the property that there exists a constant $C>0$ such that for each $I'\in[I^-,I^+]$ there exists $I''\in[I^-,I^+]$ with $|I''-I'|>C\mu$ and $x',x''\in\tilde \Lambda_\mu$ with $I(x')=I', I(x'')=I''$, such that $S^{\tilde \Gamma_\mu}(x')=x''$. In the above, one can always choose $I''>I'$ or $I''>I'$, as one wishes. That is, there are always points whose $I$-coordinates in increased or decreased by the scattering map by $O(\mu)$.

(4)  Fix a Poincar\'e surface of section $\Sigma=\{(p, q,I,\phi, t),\, t=0\, (\textrm{mod }  2\pi) \}$, and let $F_\mu$ be the first return map to $\Sigma$.
By Theorem \ref{thm:normalpoincare} $\Lambda_0=\tilde\Lambda_0\cap\Sigma$ is a normally hyperbolic invariant manifold with boundary for $F_\mu$. The map $F_\mu$ restricted to $\Lambda_\mu$ satisfies a uniform twist condition, due to (iii).
Each $2$-dimensional  torus $\tilde{\mathcal{T}}_{I'}$ in $\tilde \Lambda_\mu$ intersects $\Sigma$ in  a $1$-dimensional $\mathcal{T}_{I'}$ in $ \Lambda_\mu$.
From Subsection \ref{subsection:scatteringreturn} it follows that $\Gamma_\mu=\tilde\Gamma_\mu\cap \Sigma$ is a homoclinic channel. The scattering map $S^{\Gamma_\mu}:U^-_\eps\to U^+_\eps$ associated to $\Gamma^\mu$ for $F_\mu$ is related to  $S^{\tilde \Gamma_\mu}$ by the relation given in Proposition \ref{prop:scatteringmapflow1}.

(5) Choose a sequence $\{\mathcal{T}_{I_j}\}_{j\in\mathbb{Z}}$ of KAM tori (primary or secondary) in $\Lambda_\mu$  with the following properties:  (a) Each leave $\mathcal{T}_{I_j}$ is within $O(\mu)$, relative to the $I$-variable, from the next leave $\mathcal{T}_{I_{j+1}}$, (b) $S^{\Gamma_\mu}$  takes each $\mathcal{T}_{I_j}$ transversally across $\mathcal{T}_{I_{j+1}}$, for all $n$. From Proposition \ref{prop:transversal}, $W^u(\mathcal{T}_{I_j})$ intersects transversally $W^s(\mathcal{T}_{I_{j+1}})$ at some point $z_j\in\Gamma_\mu$.
Thus, the set  $\{\mathcal{T}_{I_j}\}_{j\in\mathbb{Z}}$, together with the transverse homoclinic connections among consecutive tori, forms a transition chain.

(6)  Now we explain how Theorem \ref{lem:shadowing1} is applied to this problem. We have to show that the assumptions  of the theorem are met for some sequence of $2$-dimensional windows $\{D_j^-,D_j^+\}$ in $\Lambda_\eps$. Following the arguments in \cite{GideaL06}, for each $j$ we can construct a pair of windows $\hat D_j^-\subseteq U_\eps^-$, $\hat D_{j+1}^+\subseteq U_\eps^+$,  with $\hat D_j^-$ in some neighborhood of $\mathcal{T}_{I_j}$ and $\hat D_{j+1}^+$ in some neighborhood of $\mathcal{T}_{I_{j+1}}$ such that $\hat D_j^-$ is correctly aligned with $\hat D_{j+1}^+$ under $S^{\Gamma_\mu}$. The exit and entry directions of each window are $1$-dimensional.  The window  $\hat D_j^-$ is chosen to have its entry set components on some pair of tori $\mathcal{T}'_{I_j}$, $\mathcal{T}''_{I_j}$ neighboring $\mathcal{T}_{I_j}$ on both its sides in $\Lambda_\eps$, and  the window  $\hat D_{j+1}^+$ is chosen to have its exit components  set on some pair of tori $\mathcal{T}'_{I_{j+1}}$, $\mathcal{T}''_{I_{j+1}}$ neighboring $\mathcal{T}_{I_{j+1}}$ on both its sides in $\Lambda_\eps$.

Choose some positive integers $n_1^0,n_1^-,n^+_1$. Choose an open neighborhood $\mathcal{N}(\Lambda_\eps)$ of $\Lambda_\eps$ in $\Sigma$ on which the local linearization of the normally hyperbolic invariant manifold from Theorem 1 in \cite{PughS70} applies. Then for each $j$ there exists $N_j^->n_1^-$ such that $F^{-N^-_j}(z_{j})$ is contained in $\mathcal{N}(\Lambda_\eps)$, and there exists $N_{j+1}^+>n_1^+$ such that $F^{N^+_{j+1}}(z_j)$ is contained in $\mathcal{N}(\Lambda_\eps)$. Due to the compactness of $\Gamma_\mu,\Lambda_\mu$, one can choose the numbers $N_j^-, N_{j+1}^+$ uniformly bounded above by some $n_2^-,n^+_2$, respectively,  for all $j$.

Now let $D_j^-:=F^{-N^-_j}(\hat D_j^-)$ and $D_{j+1}^+:=F^{N^+_{j+1}}(\hat D_{j+1}^+)$. By construction, these sets satisfy   condition (i) of Theorem \ref{lem:shadowing1}. By Definition \ref{defn:transition_map map}, we have that $D_j^-$ is correctly aligned with $D_{j+1}^+$ under the transition map $S^{\Gamma_\mu}_{N_j^-,N_{j+1}^+}$.  This ensures condition (ii) of Theorem \ref{lem:shadowing1}.
Now using the twist condition, for each $j$ there exists a large enough $N^0_j$ such that $D_{j}^+$ is correctly aligned with
$D_{j}^-$. For this we use the fact that the exit set components of $D_{j}^+$ and the entry set components of $D_{j}^-$ lie on $\mathcal{T}'_{I_j}$, $\mathcal{T}''_{I_j}$, as in \cite{GideaR09}. This way we fulfill condition (iii) of Theorem \ref{lem:shadowing1}.  We can always choose $N_j^0$ as large as we want, and in particular $N_j^0>n^0_1$.  The number of iterates $N^0_j$ to achieve this correct alignment depends  only on the twist condition on $\Lambda_\eps$, on the sizes of the windows, and on the location of the windows $D_{j}^+, D_{j}^-$ about $\mathcal{T}_{I_j}$.  Since the sizes of the windows are uniformly bounded, we can choose the number $N^0_j$ so that they are all bounded above  by some $n^0_2$, i.e.,  $N_j^0<n^0_2$ for all $j$. At this point, given positive  integers  $n_1^0,n_1^-,n^+_1$, we have obtained positive integers $n_2^0>n_1^0$, $n_2^->n_1^-$, $n^+_2>n^+_1$, and sequences $N^0_j,N^-_j,N^+_j$ as specified by Theorem \ref{lem:shadowing1}. Then, for every $\eps>0$, there exists an orbit $\{F^N(z)\}_{n\in\mathbb{Z}}$ that visits the $\eps$-neighborhoods of the windows  $\{D_j^-,D_j^+\}$ in the prescribed order. In particular, we obtain diffusing orbits and symbolic dynamics.

We should note that  the usage of  the KAM theorem to construct correctly aligned windows in the  above argument is not necessary, as it is done here only to simplify the exposition. Instead, one can only use the averaging method  as in \cite{GideaL06} and construct the windows about the almost-invariant level sets of the averaged energy.   This alternative based on the averaging method lowers the regularity requirements from condition (i) above.

As an alternative to the diffusion mechanism  described above, we can cross the large gaps corresponding to the resonances of order 1 and 2, by following Birkhoff connecting orbits, as in \cite{GideaR09}, or Mather connecting orbits, as in \cite{GideaR11}. This mechanism allows one to get rid of the assumption (v.c) above.

\subsubsection*{The spatial circular restricted three-body problem.}

The second model is the spatial circular restricted three-body problem, in the case of the Sun-Earth system. We follow \cite{DelshamsGR10}.  This problem considers  the spatial motion of an infinitesimal mass under the gravitational influence of Sun and Earth, of masses $m_1, m_2$, respectively,  that are assumed to move on circular orbits about their center of mass. Let $\mu=m_2/(m_1+m_2)$. We can translate  the equations of motion of the infinitesimal body relative to a co-rotating frame that moves together with $m_1, m_2$,  and then describe the dynamics  by a Hamiltonian system of Hamiltonian function
\begin{equation}\label{eq:scrtbp}
H(x,y,z,p_x,p_y,p_z)=\frac{1}{2}(p_x^2+p_y^2+p_z^2)+y p_x-x
p_y-\frac{1-\mu}{r_1}-\frac{\mu}{r_2},
\end{equation} where  $r_1^2=(x-\mu)^2+y^2+z^2$ and $r_2^2=(x-\mu+1)^2+y^2+z^2$.
The system has five equilibrium points; one of them, denoted $L_1$,  lies between the two primaries. Its stability  is of
saddle $\times$ center $\times$ center type. We  focus our attention on the dynamics near $L_1$. At a given energy level, near $L_1$ we can distinguish a family of quasi-periodic orbits that are uniquely defined by their out-of-plane amplitude of the motion.
The objective is to show that there exits trajectories that move from near one such a quasi-periodic orbit to near another, in the prescribed order, and so to change the  out-of-plane amplitude of the motion near $L_1$ with zero cost. The mechanism to achieve such trajectories involves starting about  a quasi-periodic orbit about $L_1$ of some out-of-plane amplitude, moving the infinitesimal mass around one of the primaries, returning to $L_1$, and moving about  some other quasi-periodic orbit about $L_1$ of  a different out-of-plane amplitude, and so on.
Since this problem is not close to integrable,  the methods from perturbation theory do not apply.  The approach discussed below is semi-numerical.

(1) Fix an energy level $h$ slightly above the energy level of $L_1$, so that the dynamical channel between the  $m_1$-region and the $m_2$-region is open. Restrict the study of the dynamics to the $5$-dimensional energy manifold $M_h=\{H=h\}$.

(2) Compute numerically the $3$-dimensional normally hyperbolic invariant manifold $\tilde \Lambda\subseteq M_h$ by writing the Hamiltonian in   Birkhoff normal form in a neighborhood of $L_1$. By taking a high-order truncation of the normal form we obtain a  system of action-angle coordinates $(I_1,I_2,\phi_1,\phi_2)$ on $\tilde\Lambda$, where we can choose $I_1$ as the out-of-plane amplitude of a quasi-periodic orbit in $\tilde\Lambda$, and $I_2$ implicitly defined by the restriction to the energy level. As the truncated normal form is integrable, the numerically computed $\tilde \Lambda\subseteq M_h$  is filled with $2$-dimensional invariant tori $\tilde{\mathcal T}_{I_1}$ parametrized by the out-of-plane amplitude $I_1$. Dynamically, these invariant tori are only almost invariant.

(3) The Birkhoff normal form is also used to  compute  the stable and unstable manifolds $W^s(\tilde\Lambda)$, $W^u(\tilde\Lambda)$  of the normally hyperbolic invariant manifold $\tilde\Lambda$. One shows numerically that $W^s(\tilde\Lambda)$, $W^u(\tilde\Lambda)$  intersect transversally along some  homoclinic channel $\tilde\Gamma$. In particular,  one can compute the scattering map $S^{\tilde \Gamma}$ for the flow as in \cite{DelshamsMR08}. We choose some $\eps>0$ sufficiently small, and choose the $t_u,t_s$ to be the first times when $\phi^{-t_u}(\tilde\Gamma )$, $\phi^{-t_s}(\tilde\Gamma )$, respectively, land in the $\eps$-neighborhood of $\tilde\Lambda$. We then compute the corresponding transition map $S^{\tilde\Gamma}_{t^u,t^s}$ as in Definition \ref{defn:transition_map flow}.

(4)  We reduce the dimension of the problem by choosing a suitable Poincar\'e section, e.g. $\Sigma=\{\phi_2=0\}$.  Let $F$ denote the first return map to $\Sigma$. We have that $\Lambda =\Sigma\cap\tilde\Lambda $ is normally hyperbolic for $F$, and  $\Gamma =\Sigma\cap\tilde\Gamma$ is a homoclinic channel.
The intersections of the $2$-dimensional  tori  $\tilde{\mathcal T}_{I_1}$ with the  Poincar\'e section yields $1$-dimensional tori $\mathcal T_{I_1}$ in $\Lambda$, each torus  corresponding to a fixed value of $I_1$. The inner dynamics induced on the Poincar\'e section is a monotone twist map. Using the relationship between the transition map for a flow and the transition map for the return map from Proposition \ref{prop:transitionmapflow1}, we compute the transition map $S^{\Gamma }_{N^u,N^s}$ for $F$.
The image of each $1$-dimensional torus $\mathcal T_{I'_1}$  under $S^{\Gamma }_{N^u,N^s}$ is a curve along which the value of $I_1$ variable sometimes goes above and sometimes below the original value $I'_1$. See Figure \ref{fig:sm_tor16.eps}.
\begin{figure}
\centering
\includegraphics[width=0.5\textwidth]{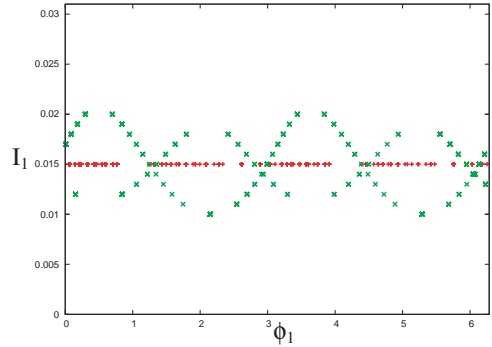}
\caption{The effect of the transition map on action level sets.}
\label{fig:sm_tor16.eps}
\end{figure}

(5) To design trajectories that visit some finite collection of $I_1$-level sets in the prescribed order, we first use the transition map $S^{\Gamma}_{N^u,N^s}$ to move between level sets. We obtain pairs of points $x_j,x_{j+1}$ with $I(x_j)=(I_1)_j$, $I(x')=(I_1)_{j+1}$, and
$S^{\Gamma}_{N^u,N^s}(x_{j})=x_{j+1}$, moving from a level set $(I_1)_{j}$ to a level set $(I_1)_{j+1}$. Using continuation, we construct a window $D^-_j$ about  the level set $(I_1)_{j}$,  and a window $D^+_{j+1}$ about  the level set  $(I_1)_{j+1}$, such that $D^-_j$ is correctly aligned with $D^+_{j+1}$ under $S^{\Gamma}_{N^u,N^s}$. Similarly, we construct  a window $D^-_{j+1}$ about  the level set $(I_1)_{j+1}$,  and a window $D^+_{j+2}$ about  a level set  $(I_1)_{j+2}$, such that $D^-_{j+1}$ is correctly aligned with $D^+_{j+2}$ under $S^{\Gamma}_{N^u,N^s}$. To align $D^+_{j+1}$ with $D^-_{j+1}$  we take some large number of iterates $N^0_j$ such that $D^+_{j+1}$ is correctly aligned with $D^-_{j+1}$ under $F^{N^0_j}$. We obtain the situation described in Theorem \ref{lem:shadowing1}. In particular, we obtain diffusing orbits for which the action variable $I_1$ increases as much as possible.

We should note that, since the correct alignment of windows is robust, the fact that the tori $\mathcal{T}_{I_1}$ are only almost-invariant
does not matter. This makes the method suitable for a computer assisted proof.

In both examples, the main advantage of using the transition map and the Theorem \ref{lem:shadowing1} is that, via these tools, one only needs to construct windows of lower dimension (the dimension of the normally hyperbolic invariant manifold) that are correctly aligned either under the transition map, or under some power of the inner map. Both these maps are defined on the lower dimensional normally hyperbolic invariant manifold. The result from Theorem \ref{lem:shadowing1} is the existence of a certain trajectory that lives in the full dimensional phase space.

\medskip
% The data information below will be filled by AIMS editorial staff
Received xxxx 20xx; revised xxxx 20xx.
\medskip

\begin{thebibliography}{99}
\def\cprime{$'$} \def\cprime{$'$}




\bibitem{Arnold64}
V.I. Arnold.
\newblock Instability of dynamical systems with several degrees of freedom.
\newblock {\em Sov. Math. Doklady}, 5:581--585, 1964.

\bibitem{BurnsG05}
Keith Burns and Marian Gidea.
\newblock {\em Differential geometry and topology. With a view to dynamical
  systems}.
\newblock Studies in Advanced Mathematics. Chapman \& Hall/CRC, Boca Raton, FL,
  2005.

\bibitem{Cresson2008}
Jacky Cresson and Christophe Guillet.
\newblock Hyperbolicity versus partial-hyperbolicity and the
  transversality-torsion phenomenon.
\newblock {\em J. Differential Equations}, 244(9):2123--2132, 2008.

\bibitem{DelshamsGR10}
A.~Delshams, M.~Gidea, and P.~Roldan.
\newblock Arnold's mechanism of diffusion in the spatial circular restricted
  three-body problem: A semi-numerical argument, 2010.

\bibitem{DelshamsLS2000}
Amadeu Delshams, Rafael de~la Llave, and Tere~M. Seara.
\newblock A geometric approach to the existence of orbits with unbounded energy
  in generic periodic perturbations by a potential of generic geodesic flows of
  {${\bf T}^2$}.
\newblock {\em Comm. Math. Phys.}, 209(2):353--392, 2000.

\bibitem{DelshamsLS2006}
Amadeu Delshams, Rafael de~la Llave, and Tere~M. Seara.
\newblock A geometric mechanism for diffusion in {H}amiltonian systems
  overcoming the large gap problem: heuristics and rigorous verification on a
  model.
\newblock {\em Mem. Amer. Math. Soc.}, 179(844):viii+141, 2006.

\bibitem{DelshamsLS08a}
Amadeu Delshams, Rafael de~la Llave, and Tere~M. Seara.
\newblock Geometric properties of the scattering map of a normally hyperbolic
  invariant manifold.
\newblock {\em Adv. Math.}, 217(3):1096--1153, 2008.

\bibitem{DelshamsGLS2008}
Amadeu Delshams, Marian Gidea, Rafael de~la Llave, and Tere~M. Seara.
\newblock Geometric approaches to the problem of instability in {H}amiltonian
  systems. {A}n informal presentation.
\newblock In {\em Hamiltonian dynamical systems and applications}, NATO Sci.
  Peace Secur. Ser. B Phys. Biophys., pages 285--336. Springer, Dordrecht,
  2008.

\bibitem{DelshamsMR08}
Amadeu Delshams, Josep Masdemont, and Pablo Rold{\'a}n.
\newblock Computing the scattering map in the spatial {H}ill's problem.
\newblock {\em Discrete Contin. Dyn. Syst. Ser. B}, 10(2-3):455--483, 2008.

\bibitem{Easton78}
Robert~W. Easton.
\newblock Homoclinic phenomena in {H}amiltonian systems with several degrees of
  freedom.
\newblock {\em J. Differential Equations}, 29(2):241--252, 1978.

\bibitem{Fenichel74}
N.~Fenichel.
\newblock Asymptotic stability with rate conditions.
\newblock {\em Indiana Univ. Math. J.}, 23:1109--1137, 1973/74.

\bibitem{Garcia00}
Antonio Garc{\'{\i}}a.
\newblock Transition tori near an elliptic fixed point.
\newblock {\em Discrete Contin. Dynam. Systems}, 6(2):381--392, 2000.

\bibitem{GideaL06}
Marian Gidea and Rafael de~la Llave.
\newblock Topological methods in the instability problem of {H}amiltonian
  systems.
\newblock {\em Discrete Contin. Dyn. Syst.}, 14(2):295--328, 2006.

\bibitem{GideaR03}
Marian Gidea and Clark Robinson.
\newblock Topologically crossing heteroclinic connections to invariant tori.
\newblock {\em J. Differential Equations}, 193(1):49--74, 2003.

\bibitem{GideaR09}
Marian Gidea and Clark Robinson.
\newblock Obstruction argument for transition chains of tori interspersed with
  gaps.
\newblock {\em Discrete Contin. Dyn. Syst. Ser. S}, 2(2):393--416, 2009.

\bibitem{GideaR11}
M.~Gidea, and C. Robinson.
\newblock  Diffusion along transition chains of invariant tori and Aubry-Mather sets, 2011.

\bibitem{HirschPS77}
M.W. Hirsch, C.C. Pugh, and M.~Shub.
\newblock {\em Invariant manifolds}, volume 583 of {\em Lecture Notes in Math.}
\newblock Springer-Verlag, Berlin, 1977.

\bibitem{HWZ98}
H.~Hofer, K.~Wysocki, and E.~Zehnder.
\newblock The dynamics on three-dimensional strictly convex energy surfaces.
\newblock {\em Ann. of Math. (2)}, 148(1):197--289, 1998.

\bibitem{Marco08}
Jean-Pierre Marco.
\newblock A normally hyperbolic lambda lemma with applications to diffusion.
\newblock Preprint, 2008.

\bibitem{PughS70}
Charles Pugh and Michael Shub.
\newblock Linearization of normally hyperbolic diffeomorphisms and flows.
\newblock {\em Invent. Math.}, 10:187--198, 1970.

\bibitem{Robinson1999}
Clark Robinson.
\newblock {\em Dynamical systems}.
\newblock Studies in Advanced Mathematics. CRC Press, Boca Raton, FL, second
  edition, 1999.
\newblock Stability, symbolic dynamics, and chaos.

\bibitem{GideaZ04a}
Piotr Zgliczy{\'n}ski and Marian Gidea.
\newblock Covering relations for multidimensional dynamical systems.
\newblock {\em J. Differential Equations}, 202(1):32--58, 2004.

\end{thebibliography}
\end{document}